\documentclass[reqno,12pt,oneside]{amsart}
\usepackage{hyperref}
\usepackage{amssymb,fge}

\usepackage{anysize}
\usepackage[english]{babel}
\usepackage{epic,eepic} 
\usepackage{epsfig}
\usepackage{hyperref}
\usepackage{float}
\usepackage[applemac]{inputenc}
\usepackage[parfill]{parskip} 
\usepackage[mathscr]{eucal}
\usepackage{MnSymbol}
\usepackage{calligra}

\usepackage{color}

\usepackage{tensor} 
 \usepackage{amsthm,amsfonts, amssymb}
\usepackage{enumerate}
\usepackage{anysize}
\usepackage{array}
\usepackage{enumitem}
\setlist[itemize,1]{leftmargin=\dimexpr 26pt-.065in}
\theoremstyle{plain}
\newtheorem{theorem}{Theorem}[section]

\theoremstyle{definition}
\newtheorem{definition}[theorem]{Definition}
\newtheorem{proposition}[theorem]{Proposition}
\newtheorem{remark}[theorem]{Remark}
\newtheorem{corollary}[theorem]{Corollary}
\newtheorem{lemma}[theorem]{Lemma}
\newtheorem{example}[theorem]{Example}

\renewenvironment{proof}{{\noindent \bf  Proof.}}{\qed}
\begin{document}
\title[\hfilneg Generalised Fractional Evolution Equations of Caputo Type\hfil]{{\itshape G\MakeLowercase{eneralised} F\MakeLowercase{ractional} e\MakeLowercase{volution} e\MakeLowercase{quations of }  C\MakeLowercase{aputo } t\MakeLowercase{ype} }}

\author{M. E. H\MakeLowercase{ern\'andez}-H\MakeLowercase{ern\'andez}, V. N. K\MakeLowercase{olokoltsov}, L. T\MakeLowercase{oniazzi}}

\address{M. E. Hern\'andez-Hern\'andez \newline
Department of Statistics, University of Warwick, Coventry, United Kingdom.}
\email{M.Hernandez-Hernandez@warwick.ac.uk}

\address{V. N. Kolokoltsov \newline
Department of Statistics, University of Warwick, Coventry, United Kingdom,
\newline
And Associate member  of Institute of Informatics Problems, FRC CSC RAS.}\email{V.Kolokoltsov@warwick.ac.uk}

\address{L. Toniazzi\newline
Department of Mathematics, University of Warwick, Coventry, United Kingdom. }
\email{L.Toniazzi@warwick.ac.uk}

%\thanks{ }
\subjclass[2010]{34A08, 26A33, 34A12, 60H30, 35S15,  34A05}
\keywords{Fractional evolution equation, Generalised derivatives of Caputo type,  Mittag-Leffler functions, Feller process, $\beta$-stable subordinator, Stopping time, Boundary point.}

\setcounter{page}{1}

\maketitle

\begin{abstract}
This paper is devoted to the study of generalised time-fractional evolution equations involving    Caputo  type  derivatives.  Using analytical methods and probabilistic arguments   we obtain well-posedness results   and stochastic representations for the solutions. These results encompass  known linear and non-linear equations from  classical fractional  partial differential equations such as  the  time-space-fractional diffusion equation, as well as their far reaching  extensions. \\
Meaning is given to a probabilistic generalisation of Mittag-Leffler functions.
\end{abstract}
%\tableofcontents
\section{Introduction}
\label{section1}
The main purpose of this article is to prove well-posedness and stochastic representation for the solutions of the following evolution equations
\begin{align}\nonumber
-_{t}D^{(\nu)}_{a+*}u(t,x)&= -Au(t,x)-g(t,x), &&(t,x)\in(a,b]\times \mathbb{R}^d,\\
 u(a,x)&=\phi_a(x), && x\in\mathbb{R}^d, \label{Problem1}
\end{align}
and 
\begin{align}\nonumber
-_{t}D^{(\nu)}_{a+*}u(t,x)&= -Au(t,x)-f(t,x,u(t,x)),\ &&(t,x)\in(a,b]\times \mathbb{R}^d,\\
 u(a,x)&=\phi_a(x), && x\in\mathbb{R}^d, \label{Problem2}
\end{align}
where $-_{t}D^{(\nu)}_{a+*}$ is a generalised differential operator of  Caputo type of order less than $1$ acting on the time variable $t\in [a,b]$ (as introduced in \cite{KVFDE}), $A$ is the (infinitesimal) generator of a Feller semigroup  on $C_\infty(\mathbb{R}^d)$ acting on the variable $x\in\mathbb R^d$, $\phi_a$ belongs to the domain  of the generator $A$ (denoted by $ Dom(A)$), $g:[a,b]\times \mathbb{R}^d\to \mathbb R$ is a   bounded measurable function, and  $f:[a,b]\times \mathbb R^{d}\times \mathbb R\to \mathbb R$ is a non-linear function satisfying a certain Lipschitz condition.\\

Since Caputo derivatives of order $\beta \in (0,1)$ are special cases of the operators $-_{t}D_{a+*}^{(\nu)}$, the evolution equations   in (\ref{Problem1})-(\ref{Problem2}) include as particular cases a variety of    equations studied in the theory of fractional partial differential equations (FPDE's).  The latter equations have been successfully used for describing diffusions in disordered media, also called   \textit{anomalous diffusions}, which include both  \textit{subdiffusions}     and  \textit{superdiffusions}. Subdiffusion phenomena are usually related to  time-FPDE's,  whereas superdiffusions are related to   space-FPDE's.  We refer, e.g.,  to    \cite{Bouchaud1990}, \cite{carpinteri1997}, \cite{Wings},  \cite{FMainardi1997}, \cite{FMainardi2010}, \cite{kilbas2}, \cite{Leo01}, \cite{KV}, \cite{LoHiBe2011}  \cite{Meerschaert2012}, \cite{podlubny},  \cite{zaslavsky} \cite{Kokine} (and  references cited therein) for an account of historical notes, theory and applications of fractional calculus, as well as different analytical and numerical methods  to address both  fractional  ordinary differential equations (FODE's) and fractional partial differential equations.   
  
In the classical fractional setting,   special cases  of equation (\ref{Problem1}) include  \textit{fractional Cauchy problems}, that is  initial value problems  of the form 
 \begin{align}\nonumber
 -_{t}D_{a+*}^{\beta} u(t,x)  &= -A u (t,x), && (t,x) \in [a,b]\times  \mathbb{R}^d,\\
 u(a,x) &= \phi_a(x), && x \in \mathbb{R}^d, \quad \beta \in (0,1),  \label{E:fEv}
 \end{align}
 where $_{t}D_{a+*}^{\beta}$ stands for the Caputo derivative of order $\beta$ (acting on the variable $t$).
Equations of the type in  (\ref{E:fEv}) have been actively studied in the literature.   Amongst the standard analytical  approaches   to solve FPDE's,    \textit{the   Laplace-Fourier  transform} method    plays an important role (see, e.g.,  \cite{kai},   \cite{edwards}, \cite{kilbas2},   \cite{podlubny}, \cite{samko}, and references therein).   From a probabilistic point of view, interesting  connections have been found between the  solution of  time-FPDE's and the transition densities of time-changed Markov  processes (see for example \cite{meerspecneg16}, \cite{Bae05},  \cite{gorenflo98},  \cite{KV}, \cite{KV0},   \cite{Meerschaert2012}, \cite{non1990}).  For instance, a very standard example of  the equation (\ref{E:fEv}), first studied by  Schneider and Wyss   \cite{Wyss22} and Kochubei \cite{Ko13} (see also   \cite{Bouchaud1990}, \cite{FMainardi1997}, \cite{Meerschaert2012} and references therein), is given by  the time-\textit{fractional diffusion   equation}, where    $-A= \,-\, \frac{1}{2} \Delta$,  $\Delta$ being the Laplace operator.  The work in \cite{Bae01} provides strong solutions for $A$ being the generator of a Feller process.   The work in \cite{Leo13} provides strong solutions for $A$ being the generator of a Pearson diffusion on an interval. In these cases the    fundamental  solution (or Green function)  corresponds to the probability density of a self-similar non-Markovian stochastic process, given by  the time-changed  transition probability function of the diffusion associated with $A$ by the hitting time of a $\beta$-stable subordinator.  \\    
 An example of equation  (\ref{E:fEv}) (with a potential),  was studied in  \cite{kochubei}, wherein the authors determined the fundamental solution of the non-homogeneous Cauchy problem associated with    the   second-order differential operator with variable coefficients given by 
 \[ 
 A =\sum_{i,j}^d a_{ij} (x) \frac{\partial^2}{\partial x_i \partial x_j} + \sum_{j=1}^d b_{j}(x) \frac{\partial}{\partial x_j} + c(x).
\]
The well-posedness of the (abstract) Cauchy problem (\ref{E:fEv}) for $A$ being a closed operator in a Banach space was studied in \cite{Bazh98}. Moreover, evolution equations of the type (\ref{E:fEv}) arise, for example, as the limiting evolution of an uncoupled and properly scaled \textit{continuous time random  walk} (CTRW)  with the waiting times in the \textit{domain of attraction of $\beta-$stable laws}.   This  probabilistic model and some of its extensions  have been  widely studied (see, e.g., \cite{Meerschaert2012}, \cite{scalas}, \cite{KV0}, and references therein). The authors in     \cite{KM-1} addressed the regularity of  the non-homogeneous time-space fractional linear equation 
\begin{align*}\nonumber 
 _{t\,\,}\!D_{0+*}^{\beta} u(t,x)  &= -c (-\Delta)^{\alpha/2} u (t,x) + g(t, x), && x \in \mathbb{R}^d, \,\, t \ge0 ,\\ u(0,x) &= \phi_0(x), && x \in \mathbb{R}^d 
 \end{align*}
as well as the well-posedness for the fractional Hamilton-Jacobi-Bellman (HJB) type  equation
\begin{align*}\nonumber 
 _{t\,\,}\!D_{0+*}^{\beta} u(t,x)  &= -c (-\Delta)^{\alpha/2} u (t,x) + H(t, x, \nabla u(t,x)), && x \in \mathbb{R}^d, \,\,t \ge 0,\\ u(0,x) &= \phi_0(x), && x \in \mathbb{R}^d,
 \end{align*}
 for $\beta \in (0,1)$, $\alpha \in (1,2]$  and a positive constant $c > 0$.  \\

Using the results presented here, we are able to deduce some of the  results known for the previous cases, as well as to extend the analysis to more general situations (see, e.g., Section \ref{section2.2} for some possible choices of concrete operators $-_{t}D^{(\nu)}_{a-*}$).  \\
We will first show the well-posedness of problem (\ref{Problem1}) (for two notions of solution) and the stochastic  representation for both notions of solution (see Theorem \ref{MThmCapSR}). The stochastic representation for the solution $u$, will be given by
\begin{equation}
u(t,x)=\mathbf E\left[\phi_a \left( X^{x,A}(\tau^{(\nu)}_a(t))\right)+\int_0^{\tau^{(\nu)}_a(t)}g(X^{t,(\nu)}_{a+*}(s), X^{x,A}(s))d s\right],
\label{SRintro}
\end{equation}

where $\{X^{t,(\nu)}_{a+*}(s)\}_{s\ge 0}$ is the \emph{decreasing } $[a,b]$-valued stochastic process generated by $-_{t}D^{(\nu)}_{a+*}$ started at $t\in[a,b]$, $\{ X^{x,A}(s)\}_{s\ge 0}$ is the stochastic process generated by $A$ started at $x\in\mathbb R^d$, $\tau^{(\nu)}_a(t)$ is the first time $\{X^{t,(\nu)}_{a+*}(s)\}_{s\ge 0}$ hits $\{a\}$. Note that the stochastic representation (\ref{SRintro})  features the (time-changed) process $\{X^{x,A}(\tau^{(\nu)}_a(t))\}_{t\ge 0}$.\\
 For $A$ bounded  and a stronger assumption on the function $\nu$ (see assumption (H1b)), we will  give the series representation to the solution of problem (\ref{Problem1})
\begin{equation}
 u(t,x)=\sum_{n=0}^\infty ((AI^{(\nu)}_{a+})^n\phi_a)(t,x)+\sum_{n=0}^\infty ((AI^{(\nu)}_{a+})^n  I^{(\nu)}_{a+} g)(t,x),
\label{seriesintro}
\end{equation}

where $I^{(\nu)}_{a+}$ is the potential operator of the semigroup generated by the (generalised) RL fractional operator $-_{t}D^{(\nu)}_{a+}$ (see Theorem \ref{CapSeriesR}). The series in (\ref{seriesintro}) provides a generalisation of a certain class of Mittag-Leffler functions. To see this take $A=\lambda$, $\lambda \in \mathbb R$, $a=0$ and $-_{t}D^{(\nu)}_{a+*}=-_{t}D^{\beta}_{a+*}$, the Caputo derivative of order $\beta\in (0,1)$, then $I^{(\nu)}_{a+}=I^{\beta}_{a+}$,  the RL fractional integral of order $\beta$, and 
\[
 u(t,x)=\phi_a(x) E_\beta(\lambda t^\beta)+\int_0^t g(t-y,x)\beta t^{\beta-1}\frac{d}{d y}E_\beta(\lambda y^\beta) d y,
\]
where $E_\beta( z):=\left(\sum_{n=0}^\infty\frac{ z^n}{\Gamma(\beta n+1)}\right)$ (see \cite[Theorem 7.2]{kai} for example). By approximating  the generator of a Feller process $A$ with bounded operators (namely the Yosida approxiamtion) we show  the convergence of the series representation (\ref{seriesintro}) to the stochastic stochastic representation  (\ref{SRintro}) for the operator $A$ (see Theorem \ref{convgenerCaputo}).\\

As for the non-linear problem (\ref{Problem2}), we study the well-posedness following a similar strategy to the one used for the  non-linear equation studied by the authors in \cite{KV-2}. Namely, by means of the   the integral representation (\textit{mild form}) of the solution to the   linear problem (\ref{Problem1}), we reduce the analysis of (\ref{Problem2}) to a fixed point problem for a suitable  linear operator (see Theorem \ref{MThmNonlinear}). 
Let us mention that, even though in this work we do not include the HJB type case, our results for the generalised  non-linear equation (\ref{Problem2}) can be used to extend the well-posedness for   the corresponding equations of  HJB type. \\

The results concerning the series representations (\ref{seriesintro}) of the solutions to the linear evolution equation (\ref{Problem1}) and the well-posedness of the non-linear evolution equation (\ref{Problem2}) rely on the bounds in Theorem \ref{T:n-iter}. Theorem \ref{T:n-iter} is a consequence of assumption (H1b), which implies that for every $t,y\in [a,b]$, $s\in\mathbb R^+$, $\mathbf P[X^{t,(\nu)}_{a+*}(s)\ge y]\le \mathbf P[X^{t,\beta}_+(s)\ge y]$ where $X^{t,\beta}_+$ is some inverted $\beta$-stable subordinator  of order $\beta\in(0,1)$.\\

Let us briefly describe the two notions of solution used in this work for problem (\ref{Problem1}). We call $u\in C_{\infty}([a,b]\times \mathbb R^d)$ a \emph{solution in the domain of the generator} for problem  (\ref{Problem1}) with $g\in C_{\infty}([a,b]\times\mathbb R^d)$, $\phi_a\in Dom(A)$, if $u$ satisfies the two equalities in (\ref{Problem1}) and   $ u \in Dom(-D^{(\nu)}_{a+*}+A)$, the domain of the generator $-D^{(\nu)}_{a+*}+A$.\\
 This notion of solution is quite natural from the point of view of semigroup theory. To see this consider a strongly continuous semigroup $\{T_s\}_{s\ge 0}$ acting on a Banach space $B$, let  $G$ be its generator and $Dom(G)$ the domain of $G$. Suppose now that the potential operator $(-G)^{-1}$ is bounded on $B$, then $(-G)^{-1}:B\to  Dom(G)$ is a bijection and $G(-G)^{-1}g=-g$ (see \cite[Theorem 1.1']{dynkin1965}).  By viewing problem (\ref{Problem1}) as a Dirichlet problem of the form 
\[
Gu(t,x)=-g(t,x),\text{ in }(a,b]\times\mathbb R^d,\quad u(a,x)=\phi_a(x)\text{ on }\{a\}\times\mathbb R^d,
\]
  for $G=( -_{t}D^{(\nu)}_{a+}+A)$, where  $-D^{(\nu)}_{a+}$ is the generalised  Riemann-Liouville (RL)  fractional derivative, $\phi_a=0$, we will see that  $(-G)^{-1}$ is bounded. From the RL case we extend the definition to the Caputo case. Of course such definition of solution does not allow to choose the boundary condition $\phi_a$, as  $u(a,\cdot)$ is determined by the choice of $g\in B$.\\
The second notion of solution overcomes this issue. Roughly speaking, a function $u\in B([a,b]\times \mathbb R^d)$ is said to be a \emph{generalised solution} to  problem (\ref{Problem1}) if $u$ is the point-wise  limit of a certain sequence of solutions in the domain of the generator. The stochastic representation of  solutions in the domain of the generator allows us to pass to the limit and obtain well-posedness along with the stochastic representation (\ref{SRintro}) of the generalised solution.\\
All results of this work concerning solutions in the domain of the generator hold true (with no change in the proofs) if we substitute $\mathbb R^d$ with the closure of an open subset of $\mathbb R^d$, call it $X$, and we let $A$ be the generator of a Feller semigroup on $C_\infty(X)$. 

The paper is organized as follows. Section \ref{section2} sets  standard notation and gives a quick review about generalised Caputo type operators of order less than 1.  Section \ref{section3} introduces the generalised RL integral operator $I^{(\nu)}_{a+}$. Section \ref{section4} focuses on the well-posedness results for the equation (\ref{Problem1}) along with providing stochastic and series representations for the solutions. Section \ref{section6} deals with  the well-posedness of the  non-linear equation  (\ref{Problem2}).

\section{Preliminaries}\label{section2}
 
 \subsection{Notation}\label{section2.1}
Let $\mathbb{N}$ and $\mathbb{R}^d$  be the set of positive integers   and  the $d$-dimensional Euclidean space, $d\in\mathbb N$, respectively.    \\
For any subset $A \subset \mathbb{R}^d$,  we define the standard sets of functions
\begin{align*}
B(A):=&\{f: A\to \mathbb R:\text{ $f$ is bounded and Borel measurable}\},\\
C(A):=&\{f\in B(A):\text{ $f$ is continuous}\},\\
C_{\infty}(A):=&\{f \in C(A):\, f \text{ vanishes at infinity}\}.
\label{eq:}
\end{align*}
All these spaces are equipped with the usual sup-norm $\|\cdot\|$, making them Banach spaces.    For an open set $A\subset\mathbb R^d$ we define
\begin{align*}
C^k(A):=&\{f\in C(A): D^\gamma f\in C(A),\ \forall |\gamma|\le k\},\\
C^k_\infty(\bar A):=&\{f\in C_\infty(A): D^\gamma f\in C_\infty(A)\ \& \ D^\gamma f \text{ is uniformly continuous on }A,\ \forall |\gamma|\le k\},\\
C^\infty(\bar A):=&\cap_{k=1}^\infty C^k(\bar A)\quad\text{and}\quad C^\infty_\infty(\bar A):=\cap_{k=1}^\infty C^k_\infty(\bar A),
\end{align*}
where $\gamma$ is a multi-index, $D^\gamma$ the associated integer-order derivative operator, $\bar A$ denotes the closure of $A$, and for the last three spaces of continuous functions we identify the functions on $A$ along with their partial derivatives with their unique continuous extension to $\bar A$. If $A$ is compact we write $C^k(\bar A)=C^k_\infty(\bar A)$. \\
Two special spaces of continuous function will be of interest to us, namely
\begin{align*} 
C_a([a,b]):=&\{f\in C([a,b]): f(a)=0\},\quad \text{and } \\
C_{a,\infty}([a,b]\times X):=&\{f \in C_\infty([a,b]\times X): f(a,x)=0 \ \forall x\in X\},
\end{align*}
for $X\subset \mathbb R^d$, both equipped with the supremum norm, turning them into Banach spaces.\\
 When we write $\| f\|$ for some real-valued function $f:X\to \mathbb R$ we mean the supremum norm of $f$ over its domain.
If $L$ is a linear operator acting on a subset of a Banach space $B$ to a Banach space $\tilde B$, we denote by $Dom(L)$ the domain of $L$. If $L$ is bounded we denote its operator norm by $\|L\|$.\\
Notation $\Gamma (z)$ and $B(\alpha,\beta)$ stands for the Gamma and the Beta function, respectively.  For all $\alpha, \beta > 0$, the Beta function is defined by
\begin{equation*}
B(\alpha,\beta) := \int_0^1 u^{\alpha-1} (1-u)^{\beta-1} du.
\end{equation*} 
We shall use the following rather standard  identities
\begin{equation}\label{BetaG}
\Gamma(z+1) = z \Gamma (z), \quad \quad
B(\alpha, \beta) = \frac{\Gamma (\alpha)\Gamma (\beta)}{\Gamma (\alpha + \beta)},
\end{equation} 
and the inequality \begin{equation}\label{dragomir}
\Gamma (n a) > (n-1)! a^{2(n-1)} \big ( \Gamma (a) \big )^n,
\end{equation}
for  $n \in \mathbb{N}$ and $ a > 0$.
   Letters $\mathbf{P}$ and $\mathbf{E}$  are reserved for the probability and the mathematical expectation, respectively.  We will   use  the lower case letter $s$ as the time variable when indexing stochastic processes or semigroups (the letter $t$ will generally be used to denote the starting point of a process on $[a,b]$).  For a stochastic process $\{X^{z}(s)\}_{s\ge 0}$ the superscript $z$   means that the process starts at $z$. The notation $ \mathbf{E} \left[  f\left(X^{z} (s) \right ) \right]$ and   $\mathbf{E} \left[  f \left( X (s) \right )  | X(0) = z\right]$ are used interchangeably.\\ 
 For a topological space $X$ we write $\mathcal B(X)$ to denote its Borel $\sigma$-algebra.
 All the stochastic processes $\{X^z(s)\}_{s\ge 0}$ considered in this paper are assumed to be defined on some complete filtered probability space $(\Omega, \mathcal{F},\{\mathcal F_s\}_{s\ge0}, \mathbf{P})$ such that $\sigma(X^z(s))\subset \mathcal F_s$ for each $s\ge 0$, where $\sigma(X^z(s))$ is the smallest $\sigma$-algebra generated by $X^z(s)$. The notation a.e. stands for almost everywhere with respect to Lebesgue measure.
 
  \subsection{Feller processes}  Let  $\{T_s\}_{s \ge 0}$ be a strongly continuous semigroup  of linear bounded operators on a Banach space $(B, \|\cdot\|_B)$, i.e., $T_s:B\to B$ $s\in \mathbb R^+$, $T_{s+t}=T_sT_t \ \forall s,t\in\mathbb{R}^+$, $T_0=I$ the identity operator and $\lim_{s \to 0} \| T_s f - f \|_{B} = 0$ for all $f \in B$.  Its (infinitesimal) generator $L$ is defined  as the (possibly unbounded) operator $L : Dom(L) \subset B \to B$ given by the strong limit 
\begin{equation}\label{Gdef}
L  f  := \lim_{s \downarrow 0} \frac{T_s f- f}{ s}, \quad f\in Dom(L),
\end{equation}
where the domain of the generator $Dom(L)$   consists of those  functions $f \in B$ for which the limit in  (\ref{Gdef}) exists in the norm sense.  We denote the resolvent operator for $\lambda\ge0$ by $(\lambda-L)^{-1}$. Sometimes we use the notation $e^{Ls}=T_s$ for a semigroup $\{T_s\}_{s\ge0}$ with generator $L$. If $L$ is a closed operator and $D \subset Dom(L)$ is a subspace of $B$, then $D$ is called a \textit{core} for  $L$ if $L$ is the closure  in the graph norm of the restriction of $L$ to $D$.  If $D \subset B$ and  $T_s D  \subset D$ for all $s\ge 0$, then $D$ is said to be invariant (under the semigroup).

We say that a (time homogeneous) Markov process $Z=(Z(s))_{s\ge 0}$ taking values in $E \subset \mathbb{R}^d$   is a \textit{Feller process} (see, e.g., \cite[Section 3.6]{KV0})  if its semigroup $\{T_s\}_{s\ge0}$,  
defined  by
\[  T_s f (z)  := \mathbf{E} \left[  f \left( Z(s) \right )  | Z(0) = z\right], \quad s \ge 0, \,\,\,z \in E, \,\,\,  f \in B(E),
\]
 gives rise to a \textit{Feller semigroup}  when reduced to $C_{\infty} (E)$, i.e., it is a strongly continuous semigroup  on $C_{\infty} (E)$ and it is formed  by  positive linear contractions  ($0 \le T_s f \le 1$ whenever $0 \le f\le 1$). We denote the extension of a bounded linear operator on $C_\infty(E)$ to $B(E)$ by the same notation.

\subsection{Generalised fractional operators of Caputo type}\label{section2.2}

This section provides the basics on generalised fractional operators as introduced in \cite{KVFDE}, along with some properties and related definitions.

 Let $\nu: \mathbb{R} \times (\mathbb{R}^+ \backslash \{0\}) \to \mathbb{R}^+$ be a non-negative function of two variables. The following condition  will be always  assumed when dealing with generalised fractional operators.
\begin{quote}
\textbf{(H0)} The function $\nu(t, r)$ is  continuous as a function of two variables and continuously  differentiable in the first variable.    Furthermore,
  \begin{equation}\nonumber 
  \sup_t \int_0^{\infty}\min\{1,r\}  \nu (t,r)dr < \infty, \quad \sup_t \int_0^{\infty} \min\{1,r\} \Big | \frac{\partial}{\partial t} \nu(t,r)\Big |dr < \infty,
  \end{equation}
  and 
  \begin{equation}\nonumber   \lim_{\delta \to 0} \sup_t \int_{0 < r \le \delta} r \nu(t,r)dr = 0. 
   \end{equation}
\end{quote}

\begin{remark}
The second bound in (H0) is both a natural assumption for concrete examples (see, e.g., the L\'evy kernel (\ref{betax})) and a  natural assumption for the proof of \cite[Theorem 4.1]{KVFDE}. The last bound in (H0) is a tightness assumption also used in the proof of \cite[Theorem 4.1]{KVFDE}.
\end{remark}

 \begin{definition} Let $a,b \in \mathbb{R}$, $a < b$. For any function $\nu$ satisfying condition (H0), the operator $-D_{a+*}^{(\nu)}$, defined  by
  \begin{align}
  -D_{a+*}^{(\nu)} f(t)   &:= \int_0^{t-a} ( f(t-r) - f(t)) \nu(t,r)dr + (f(a) - f(t)) \int_{t-a}^{\infty}\nu(t,r)dr, \label{genC}
  \end{align}
  $ t\in(a,b],$ is called  \textit{the generalised Caputo type operator}.\\
The operator $-D_{a+}^{(\nu)}$, defined  by
  \begin{align}
  -D_{a+*}^{(\nu)} f(t)   &:= \int_0^{t-a} ( f(t-r) - f(t)) \nu(t,r)dr  - f(t) \int_{t-a}^{\infty}\nu(t,r)dr, \label{genRL}
  \end{align}
  $t\in(a,b],$  is called  \textit{the generalised RL type operator}.
\end{definition}

\begin{remark} 
Note that the  operator (\ref{genC}) is well-defined at least on $C^1_{\infty}([a,\infty))$ and that the  operator (\ref{genRL}) is well-defined at least on $C^1_{\infty}([a,\infty))\cap \{f(a)=0\}$.\\
The sign $-$ in the notation $-D_{a+*}^{(\nu)}$  is introduced to comply with the standard notation of fractional derivatives. \\
The subscript $t$ will be added to operators (\ref{genC}) and (\ref{genRL}) by denoting them as $  -_{t}D_{a+*}^{(\nu)}$ and $  -_{t}D_{a+}^{(\nu)}$, respectively, if we want to emphasise the variable they act on.
\end{remark}

 \subsubsection{Special cases: \textit{ the Caputo   derivatives of order $\beta \in (0,1)$.} }
 The classical fractional Caputo derivatives    are particular cases of the operator (\ref{genC}). Namely, on regular  enough functions $f$,
\begin{equation}\label{2.7a}
 \text{if }\quad  \small{ \nu(t,r)=\, -\, \frac{1}{\Gamma(-\beta) r^{1+\beta}} }, \normalsize  \quad \beta \in (0,1), \quad   \text{ then }\quad
-D_{a+*}^{(\nu)} f (t)  =  -D_{a+*}^{\beta} f(t),
\end{equation}
where  $D_{a+*}^{\beta}$  stands for the Caputo   derivative of order $\beta \in (0,1)$. Hence, 
\begin{align*}\label{Caputo}
  D_{a+*}^{\beta}f (t) &= \frac{1}{\Gamma (- \beta)} \int_0^{t-a} \frac{f(t-r) - f(t) }{r^{1+\beta}} dr - \frac{f(a) - f(t)}{ \Gamma (1-\beta) (t-a)^{\beta}}, \quad  \,\, \beta \in (0,1).
  \end{align*}
  For $\beta \in (0,1)$ and smooth enough functions $f$,  the expression in    (\ref{genC})  coincides with the standard analytical definition (\textit{Riemann-Liouville approach}) which is given in terms of the \textit{Riemann-Liouville fractional integral operator} and the standard differential operator of integer order (see, e.g.,  \cite{kai}, \cite{podlubny}, \cite{samko} and references therein).  \\
Other particular cases include the \textit{fractional derivatives of variable  order} $-D_{a+*}^{(\nu)}  \equiv \,- \, D_{a+*}^{\beta(t)} $, which  are obtained by taking $\nu$ as the function   
\begin{equation}\label{betax}
 \nu(t,r) =  \,- \, \frac{1}{\Gamma(-\beta(t))  r^{1+ \beta(t)} }
 \end{equation}
 with a suitable function $\beta : \mathbb{R} \to (0,1)$ (see \cite{KV-1}). 
Even more generally,   these operators include the  \emph{generalised distributed order fractional derivatives}:
\begin{equation}\label{Gdistributed}
-D_{a+*}^{(\nu)}f(t)\,=\, -\, \int_{-\infty}^{\infty} \omega(s,t) D_{a+*}^{\beta(s,t)} f(t)\, \mu (ds),   
\end{equation}
 where $\omega:\mathbb R\times [a,b]\to\mathbb R^+$ is a differentiable function in the second variable such that
 \[ 
\nu(t,r)= - \int_{-\infty}^{\infty} \omega (s,t) \frac{\mu(ds)}{\Gamma(-\beta(s,t))r^{1+\beta(s,t)}}
\] is a function satisfying condition (H0). In the classical fractional framework, particular cases of (\ref{Gdistributed}) have been studied for example in \cite{distributed1}, \cite{distributed2}. Let us mention that tempered L\'evy kernels of the form
 \begin{equation*}\label{2.7c}
  \small{ \nu(t,r)=\, -\, \frac{e^{-\lambda r}}{\Gamma(-\beta) r^{1+\beta}} }, \quad \beta \in (0,1),\ \lambda>0, 
\end{equation*}
fall under the assumptions (H0). Tempered L\'evy kernels are actively studied, see for example \cite{Meetemp}, \cite{WyTemp}.

 \subsubsection{Probabilistic interpretation and basic results}\label{probinterp}
For $\nu$ satisfying (H0), we define the following processes:
\begin{definition}\label{3processes}
\begin{enumerate}
	\item [(i)] We denote by $X^{t,(\nu)}_{+}:=\{X^{t,(\nu)}_{+} (s)\}_{s \ge 0}$ the Feller process started at $t\in [a,b]$ induced by the semigroup $\{T^{(\nu)+}_s\}_{s\ge 0}$ generated by the operator (\ref{Adecreasing}) (below)  on the space $C_{\infty}(\mathbb (-\infty,b])$ with core $C^1_{\infty}((-\infty,b])\cap \{f:-D_{+}^{(\nu)}f\in C_\infty((-\infty,b])\}$ (see \cite[Theorem 5.1.1]{KV0}).
		\item [(ii)] 
We denote by $X^{t,(\nu)}_{a+*}:=\{X^{t,(\nu)}_{a+*} (s)\}_{s \ge 0}$ the Feller process started at $t\in [a,b]$ induced by the semigroup $\{T^{(\nu)a+*}_s\}_{s\ge 0}$ generated by the operator (\ref{genC}) on the space $C([a,b])$ with core $C^1([a,b])$ (see \cite[Theorem 4.1]{KVFDE}).
		
			\item [(iii)] We denote by $X^{t,(\nu)}_{a+}:=\{X^{t,(\nu)}_{a+} (s)\}_{s \ge 0}$ the sub-Feller process started at $t\in (a,b]$ induced by the semigroup $\{T^{(\nu)a+}_s\}_{s\ge 0}$ generated by the operator (\ref{genRL}) on the space $C_{a} ([a,b])$ with core $C^1([a,b])\cap C_{a} ([a,b])$ (this follows from a simple modification of \cite[Theorem 4.1]{KVFDE}).
\end{enumerate}
\end{definition}
 The operator  (\ref{genC}) was   introduced in \cite{KVFDE} as  a probabilistic extension  of the classical fractional derivatives when applied to sufficiently regular functions. It can be seen  as the  generator of an interrupted  Feller processes. The  generator of the  \textit{decreasing} Feller process $X^{t,(\nu)}_{+} $ is given by
 \begin{equation}\label{Adecreasing}
  -D_+^{(\nu)}f(t) = \int_0^{\infty} \left (f(t-r) - f(r)  \right ) \nu(t,r)dr, 
  \end{equation} 
  and the process $X^{t,(\nu)}_{a+*}$   with generator $(\ref{genC})$ is obtained by absorbing at the point $a$ the process $X^{t,(\nu)}_{+}$      on its first attempt to leave the interval $(a,b]$. The process $X^{t,(\nu)}_{a+}$   with generator $(\ref{genRL})$ is obtained by killing the process $X^{t,(\nu)}_{+}$      on its first attempt to leave the interval $(a,b]$.
\begin{remark}
Since we are interested in the solutions to  differential equations on finite time intervals,  we only  consider the operators $-D_{a+*}^{(\nu)}$ and $-D_{a+}^{(\nu)}$ acting on functions defined  on the  interval $[a,b]$ instead of $[a, \infty)$, as was done originally in \cite[Theorem 4.1]{KVFDE}. 
\end{remark}

\begin{remark} 
 If  $a=-\infty$, then the operator $-D_{-\infty+*}^{(\nu)}$   coincides with the operator $-D_+^{(\nu)}$ on  functions vanishing at  infinity. This operator can be seen as the left-sided generalisation of the \textit{Marchaud derivative} \cite[Formulas 5.57-5.58]{samko}. This operator is also known as the \textit{generator form of fractional derivatives}  \cite{KV0}, \cite{Meerschaert2012}.
\end{remark}
Notation $p_{s}^{(\nu)+} (r,E)$ and $p_{s}^{(\nu)a+*} (r,E)$ denote the transition probabilities (with $s$ being the time variable) for the processes $X^{t,(\nu)}_{+}$ and $X^{t,(\nu)}_{a+*}$, respectively.\\
 We collect some results in the following

\begin{proposition}\label{3processesresults}
 \begin{enumerate} 
\item [(i)] The processes $X^{t,(\nu)}_{+},$ $X^{t,(\nu)}_{a+}$ and $X^{t,(\nu)}_{a+*}$ are non-increasing and  the sets $\{X^{t,(\nu)}_{+}(s)\in(c,d)\}$, $\{X^{t,(\nu)}_{a+}(s)\in(c,d)\}$, $\{X^{t,(\nu)}_{a+*}(s)\in(c,d)\}$ have the same probability, for every $t\in(a,b]$, $a<c<d\le b$, $s\in\mathbb R^+$. In particular $p_s^{(\nu)a+*}(t,\{a\}) = p_s^{(\nu)+} (t, (-\infty,a]),\ t \in ( a,b]$.
\item [(ii)] The   law of $\tau^{(\nu)}_a(t):=\inf\{s\ge 0:X^{t,(\nu)}_{+}(s)\le a\}$ equals the law of the first exit time from the interval $(a,b]$ of the processes $X^{t,(\nu)}_{a+*}$  for each $t\in(a,b]$ (so that we will use indistinctly the same notation  $\tau^{(\nu)}_a (t)$).  
\item [(iii)] The first exit time $\tau^{(\nu)}_a(t)$ has finite expectation  and $\mathbf E[\tau_a^{(\nu)}(t)]\to0$ as $t\to a$ under  either the  assumption (H1a) or (H1b):
\begin{quote}
\textbf{(H1a)}:  There exist  $\epsilon> 0$ and  $\delta > 0$, such that  the function $\nu$ satisfies $ \nu (t,r) \ge \delta > 0$ for all $t$ and $|r| < \epsilon$, or 

\textbf{(H1b)}:  The function $\nu$ satisfies  $\nu(t,r) \ge C r^{-1-\beta}$ for some constant $C > 0$ and $\beta \in (0,1)$.
\end{quote}

    \end{enumerate}
		\end{proposition}
		\begin{proof}
		
			Part (i) is proved in Appendix \ref{proof3processesresults}. Part (ii) is implied by (i). 
			For part (iii) see \cite[Theorem 4.1]{KVFDE}.

		\end{proof}
\begin{remark}
Note that (H1b) implies (H1a).
\end{remark}

For our notion of generalised solution we will assume
\begin{quote}
\textbf{(H2)}: the measures $p_s^{(\nu)+}(t,\cdot)$ and $p^A_s(x,\cdot)$ are absolutely continuous with respect to Lebesgue measure for each $t\in [a,b],\ x\in \mathbb R^d,\ t\in\mathbb R^+$,
\end{quote}
where  $A$ is the generator of a Feller process $\{X^{x,A}(s)\}_{s\ge 0}$ on $\mathbb R^d$, $x\in\mathbb R^d$, and we denote by $p^A_s(x,\cdot)$ the law of $X^{x,A}(s)$, $s\ge 0$, $x\in\mathbb R^d$.

%%%%%%%%%%%%%%%%%%%%%%%%%%

\section{Generalised RL integral operator $I_{a+}^{(\nu)}$}\label{section3}
We use the potential operator corresponding to the generator $-D^{(\nu)}_{a+}$ as in Definition  \ref{3processes}-(iii) to define an integral operator on  $B ([a,b])$, which can be thought of as a generalisation of the RL integral operator $I_{a+}^{\beta}$ of order $\beta \in (0,1)$ (see, e.g., \cite[Definition 2.1]{kai}).   
\begin{definition}\label{D-Inu}
Let $\nu$ be a function satisfying assumption (H0) and (H1a). The operator  $I_{a+}^{(\nu)}: B([a,b]) \to B([a,b])$ defined by
\begin{equation*}
\left (I_{a+}^{(\nu)} f  \right ) (t) := \int_{(a,t]}f(y)\left( \int_0^\infty p_s^{(\nu)+}(t,dy)d s\right),\quad t>a,
\end{equation*}
and 0 for $t=a$, will be called \textit{the generalised RL fractional integral}  associated with $\nu$.
\end{definition}
The generalised fractional integral $I_{a+}^{(\nu)}$ satisfies the following:
\begin{enumerate}
\item [(i)] for the process $X^{t,(\nu)}_{+}$  we have
\[
I^{(\nu)}_{a+}f(t)=\mathbf E\left[ \int_0^{\tau^{(\nu)}_a(t)}f(X^{t,(\nu)}_{+}(s))ds\right],
\]
  which follows from Proposition \ref{3processesresults}-(i)-(ii).

\item [(ii)] For each $f \in B[a,b]$, 
\begin{align*}
\Big | \left( I_{a+}^{(\nu)} f \right) (t) \Big | \le \|f\| \sup_{t \in [a,b]} \mathbf{E} \left [ \tau^{(\nu)}_a(t) \right ].
\end{align*}
In particular, if $f = \textbf{1}$ (the constant function 1), then 
\begin{equation*}\label{Inu-1}
\left (I_{a+}^{(\nu)} \textbf{1}  \right )(t) = \int_{(a,t]} \int_0^{\infty} p_s^{(\nu)+}(t,dy)ds  = \mathbf{E} \left [ \tau_{a}^{(\nu)} (t) \right ].
\end{equation*}

\end{enumerate}
\begin{remark}
The operator $I_{a+}^{(\nu)}$ can be thought of as the left  inverse operator of the RL type operator $-D_{a+}^{(\nu)}$. Note that the RL type operator $-D_{a+}^{(\nu)}$ coincides with the Caputo type operator $-D_{a+*}^{(\nu)}$ on functions vanishing at $a$.
\end{remark}
\begin{remark}\label{psbeta}
If $\nu (x,y)$ is given by (\ref{2.7a}), then $I_{a+}^{(\nu)}$ coincides with the Riemann-Liouville integral operator $I_{a+}^{\beta} $ of order $\beta\in(0,1)$ (see, e.g., \cite[Chapter 2]{kai}). Let   $\tau_{a}^{\beta} (t)$ be the first exit time from the interval $(a,b]$ of the inverted $\beta-$stable subordinator started at $t\in(a,b]$. If $p_s^{\beta-}(t,y)$ denotes the transition densities of the $\beta-$stable subordinator, then
\begin{equation*} \label{E:ps-beta}
p_s^{\beta-} (t,y) = s^{-1/\beta} \omega_{\beta} (s^{-1/\beta} (y-t);1,1), 
\end{equation*}
where $\omega_{\beta}(\cdot;\sigma,\gamma)$ stands for the $\beta$-stable density with scaling parameter $\sigma$, skewness parameter $\gamma$ and zero location parameter (see, e.g., \cite[Equation (7.2), page 311]{KV0}). Let $p_s^{\beta+}(t,y)$ denote the transition density of the respective inverted $\beta-$stable subordinator. Then
  \begin{align}\nonumber
\int_0^{\infty} p_s^{\beta+}(t,y)ds &= \int_0^{\infty} s^{-1/\beta} \omega_{\beta} (s^{-1/\beta}(t-y); 1,1)ds\\ 
&= (t-y)^{\beta-1}  \int_0^{\infty} u^{-1/\beta} \omega_{\beta} (u^{-1/\beta};1,1)du = \frac{1}{\Gamma(\beta)} (t-y)^{\beta-1}, \label{beta-density}
\end{align}
 using  the Mellin transform of the $\beta-$stable densities $\omega_{\beta}(z;1,1)$ for the last equality (see, e.g., \cite[Theorem 2.6.3, p. 117]{zolotarev}). 
The previous yields the known results 
\begin{align*}
\Big | \left( I_{a+}^{\beta} f \right) (t) \Big | \le \frac{1}{\Gamma(\beta + 1)} \|f\| (b-a)^{\beta },
\end{align*}
and
\begin{equation*}\label{Ibeta-1}
\left (I_{a+}^{\beta} \textbf{1}  \right )(t) = \int_a^t \int_0^{\infty} p_s^{\beta+}(t,y)dsdy  = \mathbf{E} \left [ \tau_{a}^{\beta} (t) \right ]  =  \frac{(t-a)^{\beta}}{\Gamma (\beta + 1)}.
\end{equation*}
  \end{remark}

	Let $I_{a+}^{(\nu),n}$ denote the $n$-fold iteration of the operator $I_{a+}^{(\nu)}$, $n \in \mathbb{N}$.  For convention $I_{a+}^{(\nu),0} $ stands for  the identity operator.\\
	The following result shall be important for  the following sections. It provides an explicit bound for $|I_{a+}^{(\nu)}f(t)|$ under assumption (H1b).

\begin{theorem}\label{T:n-iter}
Let $\nu$ be a function satisfying assumptions (H0), (H1b). Then, 
for each $f \in B([a,b])$, 
\begin{align}
\Big |  \left( I_{a+}^{(\nu), n} f \right) (t)\Big | &\le   \|f\|_t\frac{(b-a)^{n\beta}}{ \left( \Gamma (\beta + 1)\right)^n}   \prod_{k=0}^{n-1} B (k \beta+ 1,\, \beta), \quad  n \ge 1, \label{bound}
\end{align}
where $\|f\|_t := \sup_{y \le t} |f (y)|$.
Moreover, the series
\begin{equation}\label{0series}
\sum_{n=0}^{\infty}  \left( I_{a+}^{(\nu),n} f \right ) (t)
\end{equation}
converges uniformly on $[a,b]$. \\
\end{theorem}
\begin{proof}
By definition of the generalised fractional integral
\begin{align*}
\Big | \left( I_{a+*}^{(\nu)} f\right)(t)\Big | &\le \int_0^{\infty} \left(\int_{(a,t]} | f(y) |\,p_s^{(\nu)+} (t,dy)\right)ds \\
&\le \int_0^{\infty} \left(\int_{(a,t]} \sup_{z \le y}|f(z)| p_s^{(\nu)+} (t,dy)\right)ds.
\end{align*}

Fix  $\beta\in(0,1)$ as in (H1b) and denote by $\{X^{t,\beta}_+(s)\}_{s\ge0}$ the associated inverted $\beta$-stable subordinator. By assumption (H1b) it follows from \cite[Theorem 1.5]{Zhang00} that $\mathbf{P}[ X^{t,(\nu)}_+ (s)  > y]\le \mathbf{P}[ X^{t,\beta}_{t} (s)  > y],  \  t,y\in(a,b],\ s\in\mathbb R^+ $. Therefore 
\begin{equation*} \label{E:dom}
\mathbf{E} \left [ g \left (X^{t,(\nu)}_{a+} (s)\right )  \right ] =\mathbf{E} \left [ g \left (X^{t,(\nu)}_+ (s)\right )  \right ] \le   \mathbf{E} \left [ g \left (X^{t,\beta}_+ (s)\right )  \right ]
\end{equation*} for any non-decreasing function $g\in C^1_\infty((-\infty,b])$ such that $g(t)=0, \ \forall x\le a$, where the equality holds as a consequence of the proof of Proposition \ref{3processesresults}-(i). By a standard approximation argument we obtain 
\[
\mathbf{P}[ X^{t,(\nu)}_{a+} (s)  > y]\le \mathbf{P}[ X^{t,\beta}_+ (s)  > y],\quad t,y\in(a,b],\ s\in\mathbb R^+.
\]
 Another approximation argument yields
\begin{equation}
\mathbf{E} \left [ g \left (X^{t,(\nu)}_{a+} (s)\right )  \right ]  \le   \mathbf{E} \left [ g \left (X^{t,\beta+} (s)\right )  \right ], 
\label{equaz}
\end{equation}
for any non-decreasing bounded function $g:[a,b]\to\mathbb R$. In particular (\ref{equaz}) holds for the function $g(y) =\sup_{z \le y}|f(z)|$. Hence
\begin{align}\nonumber
\Big | \left( I_{a+*}^{(\nu)} f\right)(t)\Big | &\le \int_0^{\infty}\left( \int_{(a,t]}  | f(y)| p_s^{(\nu)+} (t,dy)\right)ds \\ \nonumber
&\le \int_0^{\infty} \int_a^t \sup_{z \le y}|f(z)| p_s^{\beta+} (t,y)dy ds\\ 
&\le  \|f\|_t\int_0^{\infty} \int_a^t  p_s^{\beta+} (t,y)dyds \le  \frac{1}{\Gamma(\beta + 1)}\|f\|_t(t-a)^{\beta}, \label{firstbound}
\end{align}
To prove the inequality (\ref{bound}) we proceed by induction.   Case $n=1$ is given by  (\ref{firstbound}). 
Assume that the inequality in (\ref{bound}) holds for $n-1$. Then, using standard identities for the Beta function, the inequality in (\ref{firstbound})  and the induction hypothesis 
\begin{align}\nonumber
\Big |  \left( I_{a+}^{(\nu), n} f \right) (t)\Big |   &=    \Big |  I_{a+}^{(\nu)}     I_{a+}^{(\nu),n-1}  f (t)\Big |  \le \int_0^{\infty} \int_a^t \sup_{z \le y}   \Big |  I_{a+}^{(\nu),n-1}  f (z) \Big | p_s^{\beta+} (t,y)dy ds  \\ \nonumber
  & \le \int_0^{\infty} \int_a^t   \|f\|_y  \frac{(y-a)^{(n-1)\beta}}{(\Gamma (\beta + 1))^{n-1}}   \prod_{k=0}^{n-2} B (k \beta+ 1,\, \beta) p_s^{\beta+} (t,y)dy ds  \\ \nonumber
  & \le \|f\|_t  \frac{1}{\left( \Gamma (\beta + 1)\right)^{n-1}}   \prod_{k=0}^{n-2} B (k \beta+ 1,\, \beta) \int_0^{\infty} \int_a^t  (y-a)^{(n-1)\beta}  p_s^{\beta+} (t,y)dy ds  \\ \nonumber
   & \le \|f\|_t   \frac{1}{\left( \Gamma (\beta + 1)\right)^{n-1}}   \prod_{k=0}^{n-2} B (k \beta+ 1,\, \beta)  \int_a^t  (y-a)^{(n-1)\beta}  (t-y)^{\beta-1} \frac{1}{\Gamma (\beta+1)} dy \\ \nonumber
	& =    \|f\|_t\frac{(b-a)^{n\beta}}{ \left( \Gamma (\beta + 1)\right)^n}   \prod_{k=0}^{n-1} B (k \beta+ 1,\, \beta),\label{E:int-beta} 
 \end{align}
where  the  last inequality uses Fubini's theorem  and the equality in  (\ref{beta-density}). \\
To prove the convergence of (\ref{0series}) we use  the identity (\ref{BetaG}) and the inequality (\ref{dragomir}) to obtain that for each $n \in \mathbb{N}$ 
 \[ 
\prod_{k=0}^{n-1} B(k\beta + 1, \beta) = \frac{\left ( \Gamma (\beta)\,\,\right )^n}{n \beta  \Gamma (n\beta) }\le \frac{\left ( \,\,\Gamma (\beta)\,\,\right )^n}{n\beta (n-1)! \beta^{2(n-1)} \left ( \,\Gamma (\beta)\,\right )^n } \le \frac{1}{n! \beta^{2n}}. 
\]
Hence, 
\begin{equation*}
 \Big |\,  \left( I_{a+}^{(\nu),n} f \right ) (t) \,\Big |  \le \|f\|\left (    \frac{(b-a)^{\beta}}{\beta^2 \Gamma (\beta + 1)}   \right )^n \frac{1}{n! } =: M_n.
\end{equation*}

Since $\sum_{n=0}^{\infty} M_n$ converges, 
Weierstrass $M-$test implies the uniform convergence of (\ref{0series}) on $[a,b]$, as required.

\end{proof}
\begin{remark}
In the classical fractional setting, the $n-$fold RL integral  $I_{a+}^{\beta,n}$ has an explicit expression obtained from its semigroup property \cite[Theorem 2.2]{kai} 
\begin{equation*}\label{n-fold-RL}
 \left( I_{a+}^{\beta, n} f \right) (t)  = \left( I_{a+}^{n\beta} f \right) (t). 
 \end{equation*}
 Hence, for $f(t) = \mathbf{1}$,
 \begin{align*}
 \left( I_{a+}^{\beta, n} f \right) (t)  = \frac{1}{\Gamma ( n \beta)} \int_a^t (t-y)^{ n\beta  -1} dy = \frac{ (t-a)^{n\beta}}{\Gamma (n\beta + 1)}.
 \end{align*}
\end{remark}

%%%%%%%%%%%%%%%%%%%%%%

\section{Generalised fractional evolution equation: Linear case}\label{section4}
Using the theory of strongly continuous semigroups and the properties of the process $X^{(\nu)}_{a+*}$ (in particular Proposition \ref{3processesresults}-(iii)), we first prove the wellposedness and  stochastic  representation for two notions of solution to the problem 
\begin{align}\nonumber
(-_{t}D^{(\nu)}_{a+}+ A)u(t,x)&=-g(t,x),&& (t,x)\in(a,b]\times \mathbb{R}^d,\\
 u(a,x)&=0,&& x\in\mathbb R^d,\label{ProblemRL}
\end{align}
for $g\in B([a,b]\times \mathbb{R}^d)$, $A$ being the generator of a Feller semigroup on $C_\infty(\mathbb R^d)$, and $\nu$ satisfying  assumptions (H0) and (H1a) (see Theorem \ref{SRRL}).\\
The series representation is obtained under the additional assumptions (H1b) and  $A$ bounded (see Theorem \ref{MThmSeries}). We then show convergence of such series representation to the stochastic representation for $A$ generator of a Feller semigroup (see Theorem \ref{RLconvofseries}).

We use the following technical results whose proof is provided in  Appendix \ref{proofMThmSR}. 
\begin{theorem}\label{MThmSR}
Let $G,$ $\tilde G$ be generators of strongly continuous, uniformly bounded semigroups $T=\{T_s\}_{s\ge0},$ $\tilde T=\{\tilde T_s\}_{s\ge0}$ on $C_\infty(X),$ $ C_\infty(\tilde X)$ with domains $D=Dom(G),$ $ \tilde D=Dom(\tilde G)$, respectively, where $X,$ $ \tilde X$ are the closure of non-empty open subsets of $\mathbb{R}^n$ and  $C^\infty_\infty(X)\subset D$, $C^\infty_\infty(\tilde X)\subset \tilde D$, respectively.\\
Define 
\begin{equation*}
\mathcal{L}:=\text{Linear span of }D\tilde D,
\end{equation*}
where $D\tilde D:=\{g=f\tilde f: \ f\in D,\ \tilde f\in \tilde D\}$.\\ %check if need C^1!!!!!!!!!!!!!!!!!!!!!
Then 
\begin{enumerate}
	\item [(i)] the closure of $G+\tilde G$ in $C_\infty(X\times \tilde X)$  on the set $\mathcal{L}$ generates a uniformly bounded strongly continuous semigroup $\{\Phi_s\}_{s\ge0}$ on $C_\infty(X\times \tilde X)$,  with invariant core $\mathcal{L}$, where $\Phi_s:=T_s\tilde T_s=\tilde T_s T_s\ s\in\mathbb R^+$ (where $G$ and $T_s$  act on the $X$-variable,  $\tilde G$ and $\tilde T_s$ act on the $\tilde X$-variable, $s\in \mathbb R^+$).\\
	We denote by $L$ the generator of  $\{\Phi_s\}_{s\ge0}$.
	
\item [(ii)] If $G,$ $ \tilde G$ are generators of Feller  semigroups, then $\{\Phi_s\}_{s\ge0}$ is  a Feller semigroup.
If $G,$ $ \tilde G$ are generators of sub-Feller  semigroups, then $\{\Phi_s\}_{s\ge0}$ is  a sub-Feller semigroup.

\item [(iii)] The same statement in (i)  holds for $X=[a,b]$, $G$ acting on $C_a([a,b])$ and $\{\Phi_s\}_{s\ge0}$  acting on $C_{a,\infty}([a,b]\times \tilde X)$ instead of $C_{\infty}([a,b]\times \tilde X)$.
\end{enumerate}
\end{theorem}

\begin{remark}
Theorem \ref{MThmSR}  allows us to solve the resolvent equation
\[
Lu=\lambda u+ g, \  \lambda\in\mathbb R^+\backslash\{0\},\ g\in C_\infty(X\times \tilde X),
\]
but what we are particularly interested in is the case $\lambda=0$, which requires more care as the potential operator is not well-defined in general.
\end{remark}
The next Proposition will be used in Section \ref{section:convergenceseriesRL}.
\begin{proposition}\label{domofA} Suppose that $\tilde G$ is bounded. Then, under the assumptions of Theorem \ref{MThmSR}, 
$f\in Dom(L)$ implies that $f(\cdot, \tilde x)\in Dom(G)$ for each $\tilde x\in \tilde X$. In particular $Lf=(G+\tilde G)f$.
\end{proposition}
\begin{proof}
Let  $f\in Dom(L)$. Since $\mathcal L$ is a core for the generator $L$, there exists $ \{f_n\}_{n\in\mathbb N}\subset \mathcal L$ such that $f_n\to f$ and $(G+\tilde G)f_n=Lf_n\to Lf$. As $\tilde G$ is bounded $\tilde Gf_n\to \tilde G f$ and so $\{\tilde Gf_n\}_{n\in\mathbb N}$ is Cauchy in $C_\infty(X\times\tilde X)$. For each $\tilde x\in\tilde X$ $f_n(\cdot,\tilde x)\to f(\cdot,\tilde x)$ in $C_\infty(X)$ and $f_n(\cdot,\tilde x)\in Dom(G)$ for each $n\in\mathbb N$, by the definition of $\mathcal L$. If we show that $Gf_n(\cdot,\tilde x)$ is Cauchy in $C_\infty(X)$ we are done as $G$ is a closed operator on $C_\infty(X)$. This follows from the inequality
\begin{equation*}
|(Gf_n-Gf_m)(x,\tilde x)|\le \|Lf_n-Lf_m\|+\|\tilde Gf_n-\tilde Gf_m\|,
\end{equation*}
and by taking $n$ and $m$ large.
\end{proof}

We now  identify two independent processes associated with the semigroups $\{T_s\}_{s\ge0}$ and $\{\tilde T_s\}_{s\ge0}$ from the process on $X\times\tilde X$ induced by the semigroup $\{\Phi_s\}_{s\ge0}$ in Theorem \ref{MThmSR}-(ii).
\begin{definition}\label{coordinate}
Let $\{\Phi_s\}_{s\ge0}$ be a Feller semigroup generated as in Theorem \ref{MThmSR}-(ii) and denote by $Y^{(t,\tilde x)}:=\{Y^{(t,\tilde x)}(s)\}_{s\ge 0}$, $(t,\tilde x)\in X\times \tilde X$ the induced Feller process.\\
For each $(t,\tilde x)\in X\times \tilde X$, define the process $X^t:=\{X^t(s)\}_{s\ge 0}$ and the process  $\tilde X^{\tilde x}:=\{\tilde X^{\tilde x}(s)\}_{s\ge 0}$ to be the processes induced by the collection of probability measures  on $X$ and  on $\tilde X$ defined as
\[
P(X^t(s)\in B):=\Phi_s1(B\times \tilde X)(t,\tilde x),\quad  B\in \mathcal{B}(X),
\]
and 
\[
 P(\tilde X^{\tilde x}(s)\in\tilde B):=\Phi_s1(X\times \tilde B)(t,\tilde x),\quad \tilde B\in \mathcal{B}(\tilde X),
\]
 respectively. Define the stochastic process $\{(X^t(s),\tilde X^{\tilde x}(s))\}_{s\ge 0}$ on $X\times\tilde X$ by 
\[
\mathbf{P}(X^{t}(s)\in B, \tilde X^{\tilde x}(s)\in \tilde B ): =\Phi_s1(B\times  \tilde B)(t,\tilde x),\quad B\in \mathcal B(X),\tilde B\in \mathcal B(\tilde X).
\]
\end{definition}

\begin{corollary}\label{IndepCoor}
Let $\{\Phi_s\}_{s\ge0}$ be a Feller semigroup generated as in Theorem \ref{MThmSR}-(ii).\\
Then $Y^{(t,\tilde x)}(s)=(X^t(s),\tilde X^{\tilde x}(s)), \ s\in \mathbb R^+$, $(t,\tilde x)\in X\times\tilde X$. Moreover the processes $X^t$ and $\tilde X^{\tilde x} $ are independent and they  equal  the processes generated by $G$ and $\tilde G$ on  $C_\infty(X)$ and $C_\infty(\tilde X)$, respectively. 
\end{corollary}
\begin{proof}
The first statement is straightforward. The latter two statements follow from
\begin{equation*}
\Phi_s1(B\times  \tilde B)(t,\tilde x)=\mathbf{P}(X^{t}(s)\in B)\mathbf{P}(\tilde X^{\tilde x}(s)\in \tilde B), \ B\in \mathcal B(X),\ \tilde B\in \mathcal B(\tilde X),\ s\in\mathbb R^+.
\label{eq:}
\end{equation*}

\end{proof}

\subsection{Linear evolution equation: RL Case}
\subsubsection{Well-posedness and stochastic representation}

We drop the subscript $t$ from the operators $-_{t}D^{(\nu)}_{a+*}$ and $-_{t}D^{(\nu)}_{a+}$.\\
With respect to the notation in Theorem \ref{MThmSR}, from now on 
\begin{align}
G&=-D^{(\nu)}_{a+},&& D=Dom(-D^{(\nu)}_{a+}),&&& C_\infty(X)=C_a([a,b]), &&&&\text{ or}
\label{RLTriple}\\
G&=-D^{(\nu)}_{a+*},&& D=Dom(-D^{(\nu)}_{a+*}),&&& C_\infty(X)=C([a,b]), &&&&\text{  and}
\label{CaputoTriple}\\
\tilde G &= A, && \tilde D=Dom (A), &&& C_\infty(\tilde X)=C_\infty(\mathbb R^d), &&&& \label{FellerTriple}
\end{align}

where the triples (\ref{RLTriple}) and (\ref{CaputoTriple}) are the ones given in Definition \ref{3processes}-(i) and Definition \ref{3processes}-(ii), respectively. The triple (\ref{FellerTriple}) is any such triple arising from  a Feller semigroup on $C_\infty(\mathbb R^d)$ with $C^\infty_\infty(\mathbb R^d)\subset Dom(A)$, and we denote the corresponding  process by $X^{x,A}:=\{X^{x,A}(s)\}_{s\ge 0}$.\\ 

We will show that the potential operator $(-L)^{-1}$ of $L= -D^{(\nu)}_{a+}+A$   (as in Theorem \ref{MThmSR}-(iii)) is bounded. We will use this fact to solve problem (\ref{ProblemRL}).\\
Define the  stopping times 
\begin{equation*}
\tau_a^{Y}((t,x)):=\inf_{s}\{s\ge 0: Y^{(t,x)}(s)\notin (a,b]\times \mathbb{R}^d\},\quad \tau_a^{X}(t):=\inf_{s}\{s\ge 0: X^t(s) \notin (a,b]\},
\label{eq:}
\end{equation*}
where $Y^{(t,x)}=\{Y^{(t,x)}(s)\}_{s\ge 0}$ and $X^t=\{X^{t}(s)\}_{s\ge 0}$ are defined as in Definition \ref{coordinate}.

\begin{proposition}\label{3st}
The stopping times $\tau_a^{Y}((t,x))$, $\tau_a^{X}(t)$ and $\tau_a^{(\nu)}(t)$ have the same distribution, in particular
\begin{equation}
\mathbf E [\tau^{Y}_a((t,x))] =\mathbf E [\tau^{(\nu)}_a(t)] <\infty, 
\label{finitest}
\end{equation}
 uniformly in $(t,x)\in[a,b]\times \mathbb R^{d}$. Moreover   $\tau_a^{Y}((t,x))$ is independent of $ \{\tilde X^x(s)\}_{s\ge 0}$.
\end{proposition}
\begin{proof}
By Corollary \ref{IndepCoor} the process $X^t$ has the same distribution of $X^{t,(\nu)}_{a+*}$. In particular $X^t$ is non-increasing and 
\begin{equation*}
\mathbf{P}(X^t(s)>a)=\mathbf{P}(X^{t,(\nu)}_{+}(s)>a).
\label{eq:}
\end{equation*}
Also $\{\tau_a^{Y}((t,x))>s\}=\{X^t(s)>a\}\cap \{\tilde X^x(s)\in \mathbb{R}^d\}$ and by independence of $X^t(s)$ and $\tilde X^x(s)$ (Corollary \ref{IndepCoor}) we have
\begin{equation*}
\mathbf{P}(\tau_a^{Y}((t,x))>s)=\mathbf{P}(\{X^t(s)>a)\mathbf{P}( \tilde X^x_s\in \mathbb{R}^d)
=\mathbf{P}(X^{t,(\nu)}_{+}(s)>a)= \mathbf{P}(\tau_a^{(\nu)}(t)>s).
\end{equation*}
This proves that $\tau_a^{Y}((t,x))$, $\tau_a^{X}(t)$ and $\tau_a^{(\nu)}(t)$ have the same distribution. In particular we obtain the equality in (\ref{finitest}). \\
The inequality in (\ref{finitest}) follows from Proposition \ref{3processesresults}-(iii).\\
The last statement can be proved using  the  computations in this proof.
\end{proof}

From now on we will use the notation $\tau_{a}^{(\nu)}(t)$ for the stopping time $\tau_a^{Y}((t,x))$. In the next proposition we obtain the boundedness and the stochastic representation for  the potential operator $(-L)^{-1}$.

\begin{proposition}\label{PotRepr}
Let  $\Phi^*:=\{T_s^*\tilde T_s\}_{s\ge 0}$ be the Feller semigroup  obtained in Theorem \ref{MThmSR}-(ii)  for the triples (\ref{CaputoTriple}) and (\ref{FellerTriple}). Denote the generator of $\Phi^*$ by $L^*$.\\
 Let $\Phi:=\{T_s\tilde T_s\}_{s\ge 0}$ be the  semigroup obtained from in Theorem \ref{MThmSR}-(iii) for the triples (\ref{RLTriple}) and (\ref{FellerTriple}). Denote the generator of $\Phi$ by $L$.\\
Then  $(-L)^{-1}: C_{a,\infty}([a,b]\times \tilde X)\to C_{a,\infty}([a,b]\times \tilde X)$ is well-defined and it is bounded. Moreover the equality $(-L^*)^{-1}g=(-L)^{-1}g$ holds if $g\in C_{a,\infty}([a,b]\times \tilde X)$ and we obtain the stochastic representation
\begin{equation*}
(-L)^{-1} g(t,x)=\mathbf E\int_{0}^{\tau^{(\nu)}_a(t)} e^{As} g(X^{t,(\nu)}_+(s), x)ds.
\label{eq:}
\end{equation*}

\end{proposition}
\begin{proof}

For each $(t,x)\in[a,b]\times \mathbb R^d$, $s\in\mathbb R^+$  
\begin{equation}
T_s^*\tilde T_s \mathbf{1}(B\times \tilde B)(t,x)=T_s\tilde T_s\mathbf{1}(B\times \tilde B)(t,x)
\label{semigroupsequal}
\end{equation}
 if $a\notin B$, $B\in\mathcal B(X), \ \tilde B\in \mathcal B(\tilde X)$  from Proposition \ref{3processesresults}-(i) and Corollary \ref{IndepCoor}. Let $g\in  C_{a,\infty}([a,b]\times \mathbb R^d)$, then
\begin{align*}\nonumber
(-L^*)^{-1} g(t,x)&=\int_0^\infty T^*_s\tilde T_s g(t,x)ds\\ \nonumber
&=\mathbf E\left(\int_{0}^{\tau^{(\nu)}_a(t)}+\int_{\tau^{(\nu)}_a(t)}^\infty\right)\tilde T_s g(X^{t,(\nu)}_{a+*}(s), x)ds\\
&=\mathbf E\int_{0}^{\tau^{(\nu)}_a(t)}\tilde T_sg(X^{t,(\nu)}_{a+*}(s), x)ds+0, \label{resrepres}
\end{align*}
where we used Proposition \ref{3st}. A similar computation using (\ref{semigroupsequal}) yields 
\begin{equation}\label{reprepotop}
(-L)^{-1} g(t,x)=\mathbf E\int_{0}^{\tau^{(\nu)}_a(t)}\tilde T_s g(X^{t,(\nu)}_{a+}(s), x)ds.
\end{equation}
That $(-L)^{-1}: C_{a,\infty}([a,b]\times \tilde X)\to C_{a,\infty}([a,b]\times \tilde X)$ is well-defined and bounded follows from Proposition \ref{3processesresults}-(iii) with the representation (\ref{reprepotop}), as
\begin{equation*}
|(-L)^{-1}g(t,x)|\le\mathbf E \int_0^{\tau_a^{(\nu)}(t)}|\tilde T_sg(X^{t,(\nu)}_+(s),\cdot)(x)|ds\le \|g\|\mathbf E [\tau_a^{(\nu)}(t)]<\infty,
\label{eq:}
\end{equation*}
and noting that  $E [\tau_a^{(\nu)}(t)]\le E [\tau_a^{(\nu)}(b)]<\infty$.
\end{proof}

We are now ready to prove the well-posedness of problem (\ref{ProblemRL}) for two notions of solutions (following \cite{KV-1}) and to obtain stochastic representations for such solutions.
\begin{definition}
Let   $g\in C_{a,\infty}([a,b]\times\mathbb R^d)$. A function $u\in C_{a,\infty}([a,b]\times \mathbb R^d)$ is said to be a \emph{solution in the domain of the generator to problem} (\ref{ProblemRL}) if $u\in Dom(L)$ and $u$ satisfies the equalities in (\ref{ProblemRL}), where $L$ is the generator obtained in Theorem \ref{MThmSR}-(iii).
\end{definition}
\begin{definition}
Let $g\in B([a,b]\times\mathbb R^d)$. A function $u\in B([a,b]\times \mathbb R^d)$ is said to be a \emph{generalised solution  to problem} (\ref{ProblemRL})  if $u=\lim_{n\to\infty}u_n$ point-wise, where $u_n$ is the solution in the domain of the generator to problem (\ref{ProblemRL})  with $\{g_n\}_{n\in\mathbb N}\subset  C_{a,\infty}([a,b]\times\mathbb R^d)$,  $\lim_{n\to\infty}g_n= g$ a.e. and $\sup_n \|g_n\|<\infty$.
\end{definition}

\begin{theorem}\label{SRRL}
Let $\nu$ be a function satisfying conditions $(H0),$ $(H1a)$ and let $A$ be the generator of a Feller semigroup on $C_\infty(\mathbb R^d)$.  
\begin{enumerate}
	\item [(i)] If $g\in C_{a,\infty}([a,b]\times\mathbb R^d)$, then there exists a unique $u\in C_{a,\infty}([a,b]\times\mathbb R^d)$  solution in the domain of the generator to problem (\ref{ProblemRL}). Moreover $u$ admits the stochastic representation
\begin{equation}
u(t,x)=\mathbf E\left[ \int^{\tau^{(\nu)}_a(t)}_0 e^{As}g(X^{t,(\nu)}_{+}(s),\cdot)(x)ds\right].
\label{RLstochrep}
\end{equation}

\item [(ii)] If $g\in B([a,b]\times\mathbb R^d)$ and (H2) holds, then there exists a unique $u\in B([a,b]\times\mathbb R^d)$ generalised solution to problem (\ref{ProblemRL}). Moreover $u$ has the stochastic representation given in (\ref{RLstochrep}).
\end{enumerate}
\end{theorem}
\begin{proof}
\begin{enumerate}
	\item [(i)]

The potential operator $(-L)^{-1}$ of the semigroup $\{\Phi_s\}_{s\ge 0}$ is bounded by  Proposition \ref{PotRepr}. Hence by Theorem 1.1' in \cite{dynkin1965} $(-L)^{-1}: C_{a,\infty}([a,b]\times \mathbb R^d)\to Dom(L)$ is a bijection, and $(-L)^{-1}g$ solves the equation
\begin{equation*}
L (-L)^{-1} g(t,x)=-g(t,x), \quad (t,x)\in[a,b]\times \mathbb R^d,\ g\in C_{a,\infty}([a,b]\times \mathbb R^d),
\label{eq:}
\end{equation*}
giving the existence and uniqueness of a solution in the domain of the generator.\\
 The stochastic representation follows from Proposition \ref{PotRepr}.

	\item [(ii)] 
	Let $g\in B([a,b]\times \mathbb R^d)$ and take $\{g_n\}_{n\in\mathbb N}\subset C_{a,\infty}([a,b]\times \mathbb R^d)$ such that $g_n\to g$ a.e. as $n\to\infty$  and $\sup_n\|g_n\|_\infty<\infty$ (such sequence can be  constructed using \cite[Theorem 7-(i)-(ii), Appendix C]{Evans}). Note that condition (H2)  and $g\in  C_{a,\infty}([a,b]\times \mathbb R^d)$  imply that  
\begin{align*}
\mathbf E g(X^{t,(\nu)}_{a+}(s),X^{x,A}(s))=&\ \mathbf Eg(X^{t,(\nu)}_{a+*}(s),X^{x,A}(s))\\
=& \int_a^t \int_{\mathbb R^d} g(z,y) p_s^{(\nu)+}(t,z)  p_s^{A}(x,y)dydz.
\label{eq:}
\end{align*} \\
	Then by Dominated Convergence Theorem (DCT)  for each $(t,x)\in[a,b]\times \mathbb R^d$ 
	\[
	F_{(t,x),n}(s):=\mathbf E[ e^{As} g_n(X^{t,(\nu)}_{a+}(s),\cdot)(x)]\to \mathbf E [ e^{As} g(X^{t,(\nu)}_{a+}(s), \cdot)(x)]=:F_{(t,x)}(s),
	\]
	as $n\to\infty$. Define $G_{(t,x)}(s):=\sup_n\|g_n\| \mathbf P(\tau_a^{(\nu)}(t)>s)$. Then
\[
\sup_n|F_{(t,x),n}(s)|\le G_{(t,x)}(s),\quad \int_0^\infty G_{(t,x)}(s)ds=\sup_n\|g_n\|\mathbf E [\tau^{(\nu)}_a(t)]<\infty,
\]
and by DCT we obtain 
\begin{equation*}
\lim_{n\to\infty}\mathbf E\left[ \int^{\tau^{(\nu)}_a(t)}_0 e^{As} g_n(X^{t,(\nu)}_+(s),\cdot)(x)ds\right]=\mathbf E\left[ \int^{\tau^{(\nu)}_a(t)}_0 e^{As}g(X^{t,(\nu)}_+(s),\cdot)(x)ds\right],
\label{eq:}
\end{equation*}
which gives existence of a generalised solution, independence of the approximating sequence, hence uniqueness, and the claimed stochastic representation.
\end{enumerate}

\end{proof}

\subsubsection{Approximation by Yosida operators}

\begin{lemma}\label{YosidaApprox}
Let $L_\lambda:=\lambda L(\lambda-L)^{-1}$ be the Yosida approximation for the generator $L$ of a Feller semigroup on $C_\infty(\mathbb{R}^d)$. Let $g\in C_{\infty}([a,b]\times \mathbb R^d)$. Let $u_\lambda \in C_{a,\infty}([a,b]\times \mathbb R^d)$ be the generalised solution  to problem (\ref{ProblemRL})  with  $A=L_\lambda$. Let $u\in C_{a,\infty}([a,b]\times \mathbb R^d)$  be the generalise solution to problem (\ref{ProblemRL}), with $A=L$.\\
Then for each $t\in[a,b]$, $u_\lambda (t,x)\to u(t,x)$ as $ \lambda\to\infty$,  uniformly in $x\in\mathbb R^d$.
\end{lemma}
\begin{proof}
By  \cite[Chapter 1, Proposition 2.7]{EK} we have that for each $g\in C_{\infty}([a,b]\times\mathbb R^d)$, $t\in[a,b]$, 
\[
\|(e^{L_\lambda s}- e^{Ls})g(t,\cdot) \|_{\mathbb R^d}\to 0\quad\text{as }\lambda\to\infty,
\]
 uniformly for $s\ge 0$ in compact sets.\\
Pick the constant function $\|g\|$ as the dominating function. Then $\|e^{L_\lambda s}g(t,\cdot)(x)\|\le 1\|g(\cdot, x)\|\le \|g\|$ which implies
\begin{equation*}
\mathbf E\left[ \int_0^{\tau_a^{(\nu)}(t)} |e^{L_\lambda s}g(X^{t,(\nu)}_+(s),\cdot)(x)|d s\right]\le \|g\|E\left[ \tau_a^{(\nu)}(t)\right]<\infty,
\label{eq:}
\end{equation*}
and the result follows from the application of  DCT.
\end{proof}

%%%%%%%%%%%%%%%%

\subsubsection{Series representation }\label{section:convergenceseriesRL}
Under the additional assumptions 
\[
 \text{$A$ is bounded and $\nu$ satisfies assumption (H1b)},
\] 
we  give a series representation for the solution in the domain of the generator and the generalised solution to problem (\ref{ProblemRL}) obtained in Theorem \ref{SRRL}.\\
Once we have the series representation  we will obtain convergence of a sequence of  series representations of solutions to the stochastic representation obtained in Theorem \ref{SRRL} for $A$ the generator of a Feller semigroup on $C_\infty(\mathbb R^d)$ (see Theorem \ref{RLconvofseries} below).\\

Let us give well-posedness  and stochastic representation for the solution to the (FODE) problem
\begin{align}\nonumber
- D^{(\nu)}_{a+}u(t)&=-g(t),&& t\in(a,b],\\
 u(a)&=0,&& g\in B([a,b]).\label{ProblemRLode}
\end{align}
\begin{definition}\label{DsolRL} 
Let $g\in C_a([a,b])$. A function $u\in C_{a}([a,b])$ is a \emph{solution in the domain of the generator for  problem} (\ref{ProblemRLode})
if $u\in Dom(-D^{(\nu)}_{a+})$ and $u$ satisfies (\ref{ProblemRLode}).
\end{definition}

\begin{definition}\label{GsolRL} A function $u\in B([a,b])$ is a generalised solution to problem (\ref{ProblemRLode})
if $u=\lim_{n\to\infty} u_n$  point-wise, where $u_n$ is the solution in the domain of the generator to problem (\ref{ProblemRLode}) for $g_n\in C_a([a,b])$, $n\in\mathbb N$,  $g_n\to g$  a.e. and $\sup_{n\in\mathbb{N}}\|g_n\|<\infty$.
\end{definition}
The following is just a simpler version of Theorem \ref{SRRL}.
\begin{theorem}\label{MThmRLode}
Let $\nu$ be a function satisfying conditions (H0), (H1a). If $g\in C_a([a,b])$ there exists a unique solution in the domain of the generator $u\in C_a([a,b])$ to problem (\ref{ProblemRLode}), and $u$ has the representation  $u=I_{a+}^{(\nu)}g$.\\
Under the additional assumption  (H2), if $g\in B([a,b])$ there exists a unique $u\in B([a,b])$ generalised solution  to problem  (\ref{ProblemRLode}), also with the representation  $u=I_{a+}^{(\nu)}g$.
\end{theorem}

\begin{theorem}\label{MThmSeries}
Let $\nu$ be a function satisfying assumption (H0), (H1b).  Suppose that $A$ is  bounded. 
\begin{enumerate}
	\item If  $g \in C_{a,\infty}([a,b]\times \mathbb R^d)$ the unique solution $u \in C_{a,\infty}([a,b]\times \mathbb R^d)$ in the domain of the generator to problem (\ref{ProblemRL}) has the series representation
	\begin{equation}\label{S:AboundedRL}
u(t,x) =  \sum_{n=0}^{\infty} \left( (I^{(\nu)}_{a+} A)^n I^{(\nu)}_{a+}g \right )(t,x),
\end{equation}
where the convergence is in the sense of the norm  of $C_{a,\infty}([a,b]\times \mathbb R^d)$.

\item [(ii)] If $g \in B([a,b]\times \mathbb R^d)$ and (H2) holds, the unique generalised solution $u \in B([a,b]\times \mathbb R^d)$ to problem (\ref{ProblemRL}) has the series representation given in (\ref{S:AboundedRL}).
\end{enumerate}
\end{theorem}
\begin{proof}
Note that by Riesz-Representation Theorem (\cite[Theorem 1.7.3]{KV0}) $A$ and $I^{(\nu)}_{a+}$ commute.
\begin{enumerate}
\item [(i)]
Let  $u\in C_{a,\infty}([a,b]\times \mathbb R^d)$ be the solution in the domain of the generator to problem (\ref{ProblemRL}) obtained in Theorem \ref{SRRL}. As $A$ is bounded and $ u\in Dom(L)$ we obtain by Proposition \ref{domofA} that for each $x\in \mathbb R^d$, $ u(\cdot,x)\in Dom(-D^{(\nu)}_{a+})$, $Lu(\cdot,x)=(-D^{(\nu)}_{a+}+A)u(\cdot,x)$. Hence $ u(\cdot,x)$ solves 
\begin{equation}
-D^{(\nu)}_{a+} u(\cdot,x)=-\tilde g(\cdot, x),\quad u(a,x)=0
\label{odeforx}
\end{equation} 
where $\tilde g (\cdot, x):=Au(\cdot, x)+g (\cdot, x)\in C_a[a,b]$, as $Au(a,\cdot)=0$. Hence, by Theorem \ref{MThmRLode}, $u(\cdot,x)$ is the unique solution in the domain of the generator to problem (\ref{odeforx}) and it has the representation
 $ u(\cdot, x)=I^{(\nu)}_{a+}\tilde g (\cdot, x)$.\\
By induction, for each $N\in \mathbb{N}$
\begin{equation}
 u(t,x)=  \sum_{n=0}^{N}  \left( (I^{(\nu)}_{a+} A)^n  I_{a+}^{(\nu)} g  \right )(t,x) +   \left( (I^{(\nu)}_{a+} A)^{N+1}  u  \right )(t,x).
\label{induction}
\end{equation}

Now observe that, 
\begin{align*}
a_n(t,x):=  \left( (I^{(\nu)}_{a+} A)^n   I_{a+}^{(\nu)}g \right ) (t,x) & \le \Big | \left( I^{(\nu)}_{a+} A)^n   I_{a+}^{(\nu)}g \right ) (t,x)  \Big |\\
  & \le  \|g\| \|A\|^n \Big | \left (I_{a+}^{(\nu),n+1} \mathbf{1}  \right ) (t) \Big | =: b_n(t).
	\end{align*}
Hence Theorem \ref{T:n-iter} implies the uniform convergence of $\sum_{n=0}^{\infty}b_n(t)$, which in turn implies the uniform convergence of $\sum_{n=0}^{\infty}a_n(t,x)$. 
 Moreover 
 \[ 
\Big | \left( ( I^{(\nu)}_{a+}A)^{N+1} u  \right ) (t,x) \Big | \le \|u\| \|A\|^N \Big | I_{a+}^{(\nu),N+1} (t,x) \Big |  \to 0,\quad N\to\infty,
 \]
due to the uniform convergence of  $\sum_{n=0}^{\infty} \|A\|^n \left (I_{a+}^{(\nu),n} \textbf{1} \right ) (t)$ on $[a,b]$, again by Theorem \ref{T:n-iter}. Then, letting $N \to \infty$ in the equality (\ref{induction}) yields the result in (\ref{S:AboundedRL}). 

\item [(ii)]
Consider a sequence $\{g_n\}_{n\in\mathbb N}\subset C_{a,\infty}([a,b]\times \mathbb R^d)$ such that $g_n\to g$  a.e. and $\sup_n\|g_n\|<\infty$. Fix $(t,x)\in [a,b]\times \mathbb R^d$. By DCT we obtain 
\begin{equation}
\lim_{n\to\infty}\sum_{m=0}^\infty F_{(t,x),n}(m)=\sum_{m=0}^\infty \left ((I^{(\nu)}_{a+} A)^m   I_{a+}^{(\nu)} g \right )(t,x),
\label{convgener}
\end{equation}
 where $F_{(t,x),n}(m):= (I^{(\nu)}_{a+} A)^m   I_{a+}^{(\nu)} g_n $. To see this observe that for every $m\in\mathbb{N}$ 
$$
\lim_{n\to\infty} F_{(t,x),n}(m)=\left ((I^{(\nu)}_{a+} A)^m  I_{a+}^{(\nu)} g \right )(t,x),
$$ and  $|F_{(t,x),n}(m)|\le F_{(t,x)}(m):=\sup_n\|g_n\| \|A\|^m(I^{(\nu)}_{a+})^{m+1}\mathbf 1(t)$.\\
By part (i) of this Theorem and part (ii) of  Theorem \ref{SRRL} the limit on the left-hand-side of  (\ref{convgener}) equals the unique generalised solution to problem (\ref{ProblemRL}). 
\end{enumerate}

\end{proof}

\subsubsection{Convergence of the series representation to the stochastic representation}\label{section:convergenceofseries}

\begin{theorem}\label{RLconvofseries} Let $\nu$ be a function satisfying assumptions (H0), (H1b). Let $A_\lambda$ be the Yosida approximation for the generator of a Feller semigroup $A$ on $C_\infty(\mathbb R^d)$. Let $g\in C_{\infty}([a,b]\times \mathbb R^d) $. \\
Then for each $t\in[a,b]$ 
\begin{equation}
\sum_{n=0}^\infty(I^{(\nu)}_{a+}A_\lambda)^nI^{(\nu)}_{a+}g (t,x) \to \mathbf E\left[\int_0^{\tau^{(\nu)}_a(t)} e^{As} g (X^{t,(\nu)}_+(s),\cdot)(x)ds\right],\quad \lambda\to\infty,
\label{eq:}
\end{equation}
 uniformly in $ x\in\mathbb R^d$.

\end{theorem}
\begin{proof}
The result follows from combining Lemma \ref{YosidaApprox} with Theorem \ref{MThmSeries}.
\end{proof}

%%%%%%%%%%%%%%%
\subsection{Linear evolution equation: Caputo case}\label{section5}

We now transfer the results for the RL generalised fractional operator $-D^{(\nu)}_{a+}$  to the Caputo generalised fractional operator $-D^{(\nu)}_{a+*}$. We will indeed look at the problem 
\begin{align}\nonumber
(-_{t}D^{(\nu)}_{a+*}+A)u(t,x)&=-g(t,x), && (t,x)\in (a,b]\times \mathbb R^d,\\
 u(a,x)&=\phi_a(x), && x\in\mathbb R^d,
\label{ProblemCaputo}
\end{align}
where $g\in B([a,b]\times \mathbb R^d)$, $\phi_a\in Dom(A)\subset C_\infty(\mathbb R^d)$, $A$ is the generator of a Feller semigroup on $C_\infty(\mathbb R^d)$ with $C_\infty^\infty(\mathbb R^d)\subset Dom(A)$ and $\nu$ is a function satisfying conditions (H0), (H1a).\\
We again drop the subscript $t$  in $-_{t}D^{(\nu)}_{a+*}$.
\begin{remark}
Note that if $\phi_a\in Dom(A)$ then $u$ satisfies 
\begin{align*}
(-D^{(\nu)}_{a+\ast}+A)u(t,x)&=-g(t,x),\\
 u(a,x)&=\phi_a(x),
\end{align*}
 if and only if $\tilde u= u-\phi_a$ satisfies
\begin{align*}
(-D^{(\nu)}_{a+\ast}+A)\tilde u(t,x)&=-(g+A\phi_a)(t,x),\\
 \tilde u(a,x)&=0,
\end{align*}
 using  the fact that  $- D^{(\nu)}_{a+*} c(t,x)=0$ for all functions $c$ constant in the $t$ variable (which is an immediate consequence of Definition \ref{genC}). We indirectly use this fact to connect the results obtained in last section about RL type evolution equations to Caputo type evolution equations.
\end{remark}

\subsubsection{Well-posedness and stochastic representation}

\begin{definition}\label{domCap} Let $g\in  C([a,b]\times \mathbb R^d)$, $\phi_a\in Dom(A)$ such that $A\phi_a(x)=-g(a,x)\ \forall x\in\mathbb R^d$. A function
$u\in C([a,b]\times \mathbb R^d)$ is a \emph{solution in the domain of the generator to problem} (\ref{ProblemCaputo}) if $u=\tilde u+\phi_a$, where $\tilde u$ is a solution in the domain of the generator for problem (\ref{ProblemRL}) with $\tilde g=g+A\phi_a\in C_{a,\infty}([a,b]\times\mathbb R^d)$.
\end{definition}

\begin{definition}\label{genreCap} Let  $g\in  B([a,b]\times \mathbb R^d)$, $\phi_a\in Dom(A)$. A function
$u\in B([a,b]\times \mathbb R^d)$ is a \emph{generalised solution for problem} (\ref{ProblemCaputo}) if $u=\tilde u+\phi_a$, where $\tilde u$ is a generalised solution to  problem (\ref{ProblemRL}) for $\tilde g:=g +A\phi\in B([a,b]\times \mathbb R^d)$.
\end{definition}

\begin{theorem}\label{MThmCapSR}
Assume that $\nu$ is a function that satisfies (H0) and (H1a).  
\begin{enumerate}
	\item [(i)]  If $g\in C_{\infty}([a,b]\times \mathbb R^d)$ and $\phi_a\in Dom(A)$ such  that $A\phi_a(\cdot)=-g(a,\cdot)$, then there exists a unique solution $u\in C_\infty([a,b]\times \mathbb R^d)$ in the domain of the generator to problem  (\ref{ProblemCaputo}) and $u$ has the stochastic representation
\begin{equation}
u(t,x)=\mathbf E \left[\phi_a ( X^{x,A}(\tau_a^{(\nu)}(t)))\right]+\mathbf E\left[\int_0^{\tau_a^{(\nu)}(t)}g(X^{t,(\nu)}_{+}(s),X^{x,A}(s))d s\right].
\label{CaputoSR}
\end{equation}
\item [(ii)] If $g\in  B([a,b]\times \mathbb R^d)$, $\phi_a\in Dom(A)$ and (H2) holds, then there exists a unique  $u\in B([a,b]\times \mathbb R^d)$  generalised solution for problem (\ref{ProblemCaputo}) and $u$ has the stochastic representation  given by (\ref{CaputoSR}). 
\end{enumerate}
\end{theorem}
\begin{proof}

\begin{enumerate}
\item [(i)]
By  the assumptions on $g$ and $\phi_a$ we have that $\tilde g:=g+A\phi_a(x)\in C_{a,\infty}([a,b]\times \mathbb R^d)$, and it follows from Theorem \ref{SRRL}-(i) that a unique solution $\tilde u$ in the domain of the generator  to problem (\ref{ProblemRL}) exists.\\
The above gives existence of a  solution in the domain of the generator to problem (\ref{ProblemCaputo}) and uniqueness.\\
By Theorem \ref{SRRL} $\tilde u$ has the stochastic representation 
\begin{align*}
\tilde u(t,x)&=\mathbf E\int_0^{\tau_a^{(\nu)}(t)}\tilde g(X^{t,(\nu)}_{+}(s),X^{x,A}(s))ds\\
&=\mathbf E\int_0^{\tau_a^{(\nu)}(t)} g(X^{t,(\nu)}_{+}(s),X^{x,A}(s))ds+\mathbf E\int_0^{\tau_a^{(\nu)}(t)} A\phi_a(X^{x,A}(s))ds.
\label{eq:}
\end{align*}
Consider $u=\tilde u+ \phi_a$, then by Dynkin formula (see \cite[Theorem 3.9.4]{KV0}) we have the equality
\[
\phi_a(x) +\mathbf E\int_0^{\tau_a^{(\nu)}(t)} A\phi_a(X^{x,A}(s))ds=\mathbf E \phi_a (X^{x,A}(\tau_a^{(\nu)}(t))),
\]
and we obtain the stochastic representation in (\ref{CaputoSR}).
	\item [(ii)] As $g+A\phi_a\in B([a,b]\times \mathbb R^d)$ existence and uniqueness follows immediately from Theorem \ref{SRRL}-(ii), and we have the stochastic representation (\ref{CaputoSR}) by the same argument at the end of part (i) of this proof.
\end{enumerate}

\end{proof}

\begin{remark}
The solution in the domain of the generator $u\in C_{\infty}([a,b]\times\mathbb R^d)$  of Theorem \ref{MThmCapSR} solves problem (\ref{ProblemCaputo}), in the sense that
\begin{align*}
 L^*u(t,x)&=L\tilde u(t,x)+ A\phi_a(x)\\
&=-g(t,x)-A\phi_a(x)+A\phi_a(x)=-g(t,x),
 \end{align*}
and $u(a,x)=\tilde u(a,x)+\phi_a(x)=\phi_a(x)$, where we use the fact that $u=\tilde u+\phi_a\in Dom(L^*)$, $L^*\tilde u=L\tilde u$ and $L^*\phi_a=A\phi_a$. Here  $L^*$ is the generator obtained in Theorem \ref{MThmSR}-(ii) and $L$ is the generator obtained in Theorem \ref{MThmSR}-(iii). For the equality $L^*u=(-D^{(\nu)}_{a+\ast}+A)u$, it is in general necessary to prove smoothness properties of $u$.  
\end{remark}
\begin{remark}
As mentioned in the introduction, all results for the solution in the domain of the generator hold (with the same proofs) if $A$ is the generator of a Feller semigroup on a bounded domain such that the respective conditions of Theorem \ref{MThmSR} are satisfied. To obtain the results for the generalised  solution it is necessary to modify assumption (H2). Such stochastic representations have been obtained for example in the case of Pearson diffusions (\cite[Theorem 4.2]{Leo13}).
\end{remark}

\begin{example}
In the standard Caputo case, i.e. $-D^{(\nu)}_{a+*}=-D^{\beta}_{a+*}$, $a=0$, the generalised solution $u$ to problem (\ref{ProblemCaputo})  has the stochastic representation
\begin{align}\nonumber
u(t,x)=& \int_{\mathbb R^d}\phi_0(y)\left(\int_0^\infty p^{A}_s(x,y)\frac{t}{\beta}s^{-\frac{1}{\beta}-1}\omega_{\beta}\left(ts^{-\frac{1}{\beta}};1,1\right)ds\right)dy\\
&+\int_{\mathbb R^d}\int^t_0g(z,y)\left(\int_0^\infty\int_0^\infty \mathbf 1 (s<r) \varphi^{\beta}_{t,s} (r,z) p^{A}_s(x,y)  drds\right)dzdy,\label{caprepr}
\end{align}
where 
\begin{align*}
\varphi^{\beta}_{t,s}(r,z):=&\mathbf 1 (s<r) p^{\beta+}_s(t,z)\frac{d}{dr}\int^0_{-\infty} p^{\beta+}_{r-s}(z,\gamma)d\gamma\\
=&\mathbf 1 (s<r) s^{-\frac{1}{\beta}}\omega_\beta\left((t-z)s^{-\frac{1}{\beta}};1,1\right)\frac{z}{\beta}(r-s)^{-\frac{1}{\beta}-1}\omega_\beta\left(z(r-s)^{-\frac{1}{\beta}};1,1\right)
\end{align*}
is the joint density of $(\tau^{\beta}_0(t), X^{t,\beta}_{0+*}(s))$ (see \cite[Proposition 4.2]{KV-1}), using the notation of assumption (H2) and Remark \ref{psbeta}. To obtain the last equality we used standard change of variables and identities for the stable densities $\omega_\beta(\cdot;\cdot,\cdot)$. In the homogeneous case ($g=0$), the  representation (\ref{caprepr}) agrees with the representations found in the literature, see for example \cite[Theorem 3.1]{Bae01}.
\end{example}

\subsubsection{Series representation}

\begin{theorem}\label{CapSeriesR}
Let $\nu$ be a function satisfying conditions (H0), (H1b). Let $A$ be a bounded linear operator on $C_\infty(\mathbb R^d)$.  
\begin{enumerate}
	\item [(i)] If $g\in C_\infty([a,b]\times\mathbb R^d)$, $\phi_a\in Dom(A),\ A\phi_a(\cdot)=-g(a,\cdot),$ then the unique solution $u\in C_\infty([a,b]\times\mathbb R^d)$ in the  domain of the generator to problem  (\ref{ProblemCaputo}) has the series representation
\begin{equation}
u(t,x)=\sum_{n=0}^\infty A^n\phi_aI^{(\nu),n}_{a+}\mathbf 1(t,x) + \sum_{n=0}^\infty (I^{(\nu)}_{a+}A)^nI^{(\nu)}_{a+}g(t,x).
\label{CaputoSeries}
\end{equation}

\item [(ii)] If $g\in B([a,b]\times\mathbb R^d)$, $\phi_a\in Dom(A)$, condition  (H2) holds, then the unique generalised solution $u\in B([a,b]\times\mathbb R^d)$ to problem  (\ref{ProblemCaputo}) has the series representation (\ref{CaputoSeries}).
\end{enumerate}

\end{theorem}
\begin{proof}
\begin{enumerate}
	\item [(i)] Let $u\in C_\infty([a,b]\times\mathbb R^d)$ be the solution in the  domain of the generator to problem  (\ref{ProblemCaputo}). By Proposition \ref{domofA},  $\tilde u:=u-\phi_a\in Dom (L)\subset C_{a,\infty}([a,b]\times \mathbb R^d)$ solves 
\begin{equation}
-D^{(\nu)}_{a+*}\tilde u(t,x)=-A\tilde u(t,x)-(g(t,x)+A\phi_a(x)), \quad \tilde u(a,\cdot)=0.
\label{RLfromCAp}
\end{equation}
By the assumptions of the Theorem $\tilde g:= g+A\phi_a\in C_{a,\infty}([a,b]\times \mathbb R^d)$. Therefore by Theorem \ref{MThmSeries}-(i)  $\tilde u$ is the unique solution in the domain of the generator to problem (\ref{RLfromCAp}) and it has  the series representation
\begin{align}\nonumber
\tilde u(t,x) =&\sum_{n=0}^\infty (I^{(\nu)}_{a+}A)^nI^{(\nu)}_{a+}\tilde g(t,x)\\
=& \sum_{n=0}^\infty (I^{(\nu)}_{a+}A)^nI^{(\nu)}_{a+} g(t,x)+\sum_{n=0}^\infty (I^{(\nu)}_{a+}A)^nI^{(\nu)}_{a+}A \phi_a(t,x).
\label{rLtocaputoseries}
\end{align}
using the fact that both series in the right-hand side converge in $C_{\infty}([a,b]\times \mathbb R^d)$ by Theorem \ref{T:n-iter}.
Then $u=\tilde u+\phi_a$ has the series representation given in (\ref{CaputoSeries}).

\item [(ii)] For $g\in B([a,b]\times \mathbb R^d)$, let $\tilde u$ be the unique generalised solution to problem (\ref{ProblemRL}) with $\tilde g=g+A\phi_a$. Then by Theorem \ref{MThmSeries}-(ii) $\tilde u$ has the representation (\ref{rLtocaputoseries}), using the fact that both series in the right-hand side converge in $B([a,b]\times \mathbb R^d)$ by Theorem \ref{T:n-iter}. Then $u=\tilde u+\phi_a$ has representation (\ref{CaputoSeries}).
\end{enumerate}
\end{proof}

\begin{definition}
Let $\nu$ satisfy conditions (H0), (H1b) and let $A$ be bounded. We call $E_{(\nu)}(A(\cdot) I^{(\nu)}_{a+}\mathbf 1):B(\mathbb R^d)\to B([a,b]\times\mathbb R^d)$ \emph{the generalised Mittag-Leffler function for $A$ and $\nu$} , defined as 
\begin{equation}
\phi_a\mapsto E_{(\nu)}(A\phi_a I^{(\nu)}_{a+}\mathbf 1)(t,x):=\sum_{n=0}^\infty A^n\phi_a(x)I^{(\nu),n}_{a+} \mathbf 1(t),
\label{eqeqeq}
\end{equation}
$(t,x)\in [a,b]\times\mathbb R^d$.
\end{definition}
\begin{remark}
The function $E_{(\nu)}(A(\cdot)I^{(\nu)}_{a+}\mathbf 1)$ provides  a probabilistic  generalisation, for $\lambda=A$ bounded operator, to the Mittag-Leffler function 
\[
E_\beta(\lambda (t-a)^\beta)=\sum_{n=0}^\infty \frac{\lambda^n (t-a)^{\beta n}}{\Gamma(\beta n+1)}=\sum_{n=0}^\infty \lambda^n\phi_a(x) I^{\beta,n}_{a+} \mathbf 1(t),
\]
where $\beta\in(0,1),$   $\phi_a(\cdot)=1$. 
\end{remark}

\subsubsection{Convergence of the series representation to the stochastic representation}

\begin{theorem}\label{convgenerCaputo} Let $\nu$ be a function satisfying (H0), (H1b), and assume that (H2) holds. Let $A$ be the generator of a Feller semigroup on $C_\infty(\mathbb R^d)$  and $A_\lambda $ its Yosida approximation.\\ 
Fix $g\in C_{\infty}([a,b]\times \mathbb R^d)$ and $\phi_a\in Dom(A)$.\\
Then  for each $t\ge 0$
\[
E_{(\nu)}(A\phi_a \mathbf1^{(\nu)})(t,x)\to  \mathbf E\phi_a(X^{x,A}(\tau_a^{(\nu)}(t))),
\]
and 
\[
  \sum_{n=0}^\infty (I^{(\nu)}_{a+}A_\lambda)^nI^{(\nu)}_{a+}g(t,x)\to \mathbf E\int_0^{\tau_a^{(\nu)}(t)}e^{A s}g(X^{t,(\nu)}_+(s),\cdot)(x)d s 
\]
as $\lambda \to\infty$, uniformly in $x\in\mathbb R^d$.
\end{theorem}
\begin{proof}
Let $u_\lambda\in B([a,b]\times \mathbb R^d)$ be the generalised solution for problem (\ref{ProblemCaputo}) for $A=A_\lambda $. Let $u\in B([a,b]\times \mathbb R^d)$ be the generalised solution for problem (\ref{ProblemCaputo}) for $A=A $.\\
By Theorem \ref{MThmCapSR} 
\begin{equation}
 u_\lambda (t,x)=\mathbf E\phi_a(X^{x,A_\lambda}(\tau_a^{(\nu)}(t)))+\mathbf E\int_0^{\tau_a^{(\nu)}(t)}e^{A_\lambda s}g(X^{t,(\nu)}_+(s),\cdot)(x)d s,
\label{eqhere0}
\end{equation}
and
\begin{equation}
u(t,x)= \mathbf E\phi_a(X^{x,A}(\tau_a^{(\nu)}(t)))+\mathbf E\int_0^{\tau_a^{(\nu)}(t)}e^{A s}g(X^{t,(\nu)}_+(s),\cdot)(x)d s.
\label{eqhere}
\end{equation}
As a consequence of Theorem \ref{MThmSeries} and Theorem \ref{SRRL} the second term in (\ref{eqhere0}) equals the series representation (\ref{S:AboundedRL}) and by Theorem \ref{RLconvofseries} it converges  as required to the second term in (\ref{eqhere}).\\
The considerations above along with  Theorem \ref{CapSeriesR} imply that the first term in (\ref{eqhere0}) equals the first term on the right-hand side of (\ref{CaputoSeries}). For the first term in (\ref{eqhere0}) observe that by \cite[Chapter 1, Proposition 2.7]{EK} 
\[
e^{A_\lambda s}\phi_a(x)\to e^{As}\phi_a(x),\quad\lambda\to\infty,
\] 
uniformly in $x\in\mathbb R^d$, for each $s\ge 0$. For each $\lambda\ge 0$
\[
\mathbf E\phi_a(X^{x,A_\lambda}(\tau^{(\nu)}_a(t)))=\int_0^\infty e^{A_\lambda s}\phi_a(x)\mu^{\tau^{(\nu)}_a(t)}(d s),
\] 
by independence of $X^{x,A_\lambda}$ and $\tau_a^{(\nu)}(t)$ (Corollary \ref{3st}), where $\mu^{\tau^{(\nu)}_a(t)}(d s)$ is the law of $\tau^{(\nu)}_a(t)$. Also
\begin{equation*}
|e^{A_\lambda s}\phi_a(x)|\le \|\phi_a\|\quad\forall \lambda> 0,\quad \text{and}\quad\int_0^\infty \|\phi_a\|\mu^{\tau_a^{(\nu)}(t)}(d s)\le \|\phi_a\|,
\label{eq:}
\end{equation*}
and the result follows from the application of DCT.
\end{proof}

\begin{remark}
  Theorem \ref{convgenerCaputo} allows us to give meaning to a generalised Mittag-Leffler function for $A$ generator of a Feller semigroup on $C_\infty(\mathbb R^d)$.
\end{remark}
%%%%%%%%%%%%%%%%%%%%%%

 \section{Generalised fractional  evolution equation: Non-linear case}\label{section6}
Let us now study the well-posedness  for the  non-linear equation given in (\ref{Problem2}).  We introduce a notion of solution and then we  proceed as in   \cite{KV-2} via fixed point arguments.
\begin{definition}\label{Def1-DA}
Let  $\nu$ be  a function satisfying (H0), (H1b). A function  $u :[a,b]\times \mathbb{R}^d \to \mathbb{R}$ is said to be a generalised solution to the non-linear equation (\ref{Problem2})  if  $ u$  is a generalised solution  to the linear equation  (\ref{Problem1}) 
with $g(t,x) := f(t,x, u(t,x))$ for all $(t,x) \in [a,b]\times \mathbb{R}^d$.  
\end{definition}

\begin{lemma}\label{Lemma2-DA}
Let  $\nu$ be a function satisfying conditions (H0), (H1b). Assume that $A$ is the generator of a Feller semigroup on $C_\infty(\mathbb{R}^d)$ and $\phi_a\in Dom(A)$ and that (H2) holds.  Suppose that $f: [a,b]\times\mathbb{R}^d\times \mathbb R \to \mathbb{R}$ is a bounded measurable function. Then, a function $u\in C([a,b]\times \mathbb{R}^d)$ is a generalised solution to equation (\ref{Problem2}) if, and only if, $u$ solves the non-linear integral equation 
\begin{align}\nonumber
u(t,x) = &  \int_0^{\infty} (e^{As} \phi_a)  (x)  \mu^{\tau^{(\nu)}_a(t)}(d s) \nonumber \\ 
&+  \mathbf E\int_0^{\tau^{(\nu)}_a(t)}e^{As}f(X^{t,(\nu)}_{+}(s),\cdot,u(X^{t,(\nu)}_{+},\cdot))(x)ds  \label{Eq10-DA},
\end{align}
where   $ \mu^{\tau^{(\nu)}_a(t)}$ is the law of $\tau_a^{(\nu)}(t)$.
\end{lemma}
\begin{proof} By Definition \ref{Def1-DA},  $u \in C([a,b]\times \mathbb{R}^d)$ is a generalised solution to (\ref{Problem2}) if and only if $u$ is a generalised  solution  to the the linear equation (\ref{Problem1}) with  $g(t,x) := f(t,x, u(t,x))$. Note that if $u\in C([a,b]\times \mathbb{R}^d)$, then $g$ is a measurable and bounded  function on $[a,b]\times \mathbb{R}^d$. Hence  Theorem \ref{MThmCapSR}-(ii) yields   the integral equation  (\ref{Eq10-DA}),  as required.
\end{proof}

Using Weissenger's fixed point theorem we prove that the integral equation (\ref{Eq10-DA}) possesses a unique solution (for a given boundary  $\phi_a$) under the following additional assumption:
\begin{quote}
\textbf{(H3):}  The function $f: [a,b]\times \mathbb{R}^d\times \mathbb{R}\to \mathbb{R}$ is bounded and  fulfils the following Lipschitz condition with respect to the third variable:  for all $(t,x,y_1), (t,x,y_2) \in [a,b] \times \mathbb{R}^d\times\mathbb{R}$, 
\begin{equation}\label{EqLz}
|f(t,x,y_1) - f(t,x,y_2)| < L_f |y_1-y_2|, \end{equation}
for a  constant $L_f>0$ (independent of $t$ and $x$).
\end{quote}

\begin{theorem}\label{MThmNonlinear}
Let $[a,b] \subset \mathbb{R}$ and $\phi_a \in Dom(A)$. Suppose that  $\nu$ is a  function satisfying conditions (H0), (H1b). Suppose that (H2) holds  and  that $f$ is a function  satisfying condition (H3). Then    problem  (\ref{Problem2})  has a unique generalised solution  $u \in C([a,b]\times\mathbb{R}^d)$. 
\end{theorem}
\begin{proof}
 By Lemma \ref{Lemma2-DA},  the existence of a unique generalised solution to  (\ref{Problem2}) means the existence of a unique solution to the integral equation (\ref{Eq10-DA}). The latter equation  can be rewritten  as a fixed point problem  $u(t,x)= (\Psi u )(t,x)$ for a suitable operator $\Psi$. \\
 
\noindent \textit{Step a)}  Definition of  the operator $\Psi$. Denote by $B_{\phi_a}$ the closed convex subset of $C\left ([a,b]\times \mathbb{R}^d\right)$ consisting of functions  satisfying $f(a) = \phi_a$. This set is a metric space when endowed with the metric induced by the norm on $C\left ([a,b]\times \mathbb{R}^d\right)$.\\
Next, define the operator $\Psi$ on $B_{\phi_a}$ by
\begin{align}\nonumber 
(\Psi u)(t,x) :=&    \int_0^{\infty} (e^{As} \phi_a)  (x)  \mu^{\tau^{(\nu)}_a(t)}(d s) \nonumber \\ &+  \mathbf E\int_0^{\tau^{(\nu)}_a(t)}e^{As}f(X^{t,(\nu)}_{+},\cdot,u(X^{t,(\nu)}_{+},\cdot))(x)ds,    \quad t \in [a,b].
\end{align}
 Note that if $u \in B_{\phi_a}$, then $(\Psi u)(\cdot,x) \in C[a,b]$ for each $x\in \mathbb{R}^d$ and $(\Psi u)(t,\cdot) \in C(\mathbb{R}^d)$ for each $t \in [a,b]$. Further, $(\Psi u)(a, x) = \phi_a(x)$ as $\mu^{\tau^{(\nu)}_a(a)} ( ds) = \delta_0(ds)$. Therefore, $\Psi:   B_{\phi_a} \to B_{\phi_a}$.\\

 \vspace{0.3cm}
\noindent \textit{Step b)}  
Let $\Psi^n$ denote the n-fold iteration of the operator $\Psi$ for  $n\ge 0$, $n \in \mathbb{N}$. For convention $\Psi^0$ denotes the identity operator.  
Note that for $n=1$, the Lipschitz condition  of $f$ and the fact that $e^{As}$ is a contraction semigroup imply
\begin{align*}
\Big | \Psi u - \Psi v \Big |(t,x) & =\Big|\mathbf E\int_0^{\tau_a^{(\nu)}(t)}e^{As}(f(X^{t,(\nu)}_{+},\cdot,u(X^{t,(\nu)}_{+},\cdot)) -f(X^{t,(\nu)}_{+},\cdot,v(X^{t,(\nu)}_{+},\cdot)) (x)ds  \Big|   \\
&\le \mathbf E\int_0^{\tau_a^{(\nu)}(t)}e^{As}(|f(X^{t,(\nu)}_{+},\cdot,u(X^{t,(\nu)}_{+},\cdot)) -f(X^{t,(\nu)}_{+},\cdot,v(X^{t,(\nu)}_{+},\cdot)|) (x)ds\\
&\le L_f \|u-v\|_t I^{(\nu)}_{a+}\mathbf 1(t),
\end{align*}
where \[  \|u-v\|_t  := \sup_{z \le t} \|u(z,\cdot) - v(z,\cdot)\|, \quad t \in  [a,b], \]
and $L_f$ is the Lipschitz constant of the function $f$. Proceeding by induction we can  prove that 
\begin{align*}\label{Eq14}
| \Psi^n u(t,x)- \Psi^n v(t,x)| \le  \|u-v\|_t  L_f^n \left( I_{a+}^{(\nu), n} \mathbf{1}\right) (t),\quad \quad n\ge 0,
\end{align*} 
where   $I_{a+}^{(\nu), n}$ is the $n$th fold iteration of the generalised fractional operator $I_{a+}^{(\nu)}$.  
Moreover, by Theorem \ref{T:n-iter},  we know that
\[ \sum_{n=0}^{\infty}L_f^n \left( I_{a+}^{(\nu), n} \mathbf{1}\right) (t)  \le  \left (\frac{L_f^n (b-a)^{\beta}}{ \beta^2 \Gamma (\beta+1)} \right )^n \frac{1}{n!} =: \alpha_n.\]  
Hence, 
 \begin{equation*}\label{Eq12}
 \|\Psi^n  u - \Psi^n v\| \le \alpha_n \|u-v\|,
 \end{equation*}
for every $n \ge 0 $ and every $u,v \in B_{\phi_a}$, where $\alpha_n \ge 0$ and $\sum_{n=0}^{\infty} \alpha_n$ converges.

 Therefore, the Weissinger fixed point theorem   \cite[Theorem D.7]{kai} guarantees  the existence of a unique fixed point $u^*\in B_{\phi_a}$ to the integral equation (\ref{Eq10-DA}), which in turn implies the existence of a generalised solution to (\ref{Problem2}), as  required.   
\end{proof}

%%%%%%%%%%%%%%%%%
\section{Appendix}\label{Appendix}

\subsection{Proof of Theorem \ref{MThmSR}}\label{proofMThmSR}
(i). It is easy to show that $D\tilde D\subset C_\infty(X\times \tilde X) $, $\Phi_t:=T_t\tilde T_t$ is a well-defined continuous  linear operator on $C_\infty(X\times \tilde X)$, $\{\Phi_t\}_{t\ge 0}$ is a uniformly bounded semigroup. That $T_t\tilde T_t=\tilde T_t T_t$ follows from Riesz-Markov representation Theorem (\cite[Theorem 1.7.4]{KV0}). \\
That $\mathcal{L}$ is a dense subspace of $C_\infty(X\times \tilde X)$ follows from Stone-Weierstrass Theorem (for locally compact spaces, see \cite[44A.I]{willard70}) by taking as a sub-algebra the linear span of $C^\infty_\infty(X)C^\infty_\infty(\tilde X)\subset \mathcal L$ where
\begin{equation*}
C^\infty(X)C^\infty(\tilde X):=\{f: f=g\tilde g,\ C^\infty_\infty(X),\ \tilde g\in C^\infty_{\infty}(\tilde X) \},
\label{eq:}
\end{equation*}
as it separates points and it does not vanish on $X\times \tilde X$.\\
Let $f=\sum_{n=1}^N\lambda_ng_n\tilde g_n\in \mathcal{L}$. Then 
\begin{align*}
\|\Phi_t f-f\|\le &\sum_{n=0}^N|\lambda_n|\| T_tg_n\tilde T_t\tilde g_n- g\tilde g\|\\
\le& \sum_{n=0}^N|\lambda_n|(\| T_tg_n\tilde T_t\tilde g_n- Tg_n\tilde g\|+\| T_tg_n\tilde g_n- g\tilde g\|)\\
\le& \sum_{n=0}^N|\lambda_n|(\|T_t\|\|g_n\|\| \tilde T_t\tilde g_n- \tilde g\|+\|\tilde g_n\|\| T_tg_n- g\|),
\label{eq:}
\end{align*}
which can be made arbitrarily small by choice of $t$ small, using the strong continuity and the uniform boundedness  of $(T_t)$ and $(\tilde T_t)$.\\
As $\mathcal{L}$ is dense in $C_\infty(X\times\tilde X)$, it  follows that $\Phi_t$ strongly continuous on $C_\infty(X\times\tilde X)$.\\
The semigroup $\{\Phi_t\}_{t\ge 0}$ is invariant on $\mathcal{L}$ as $T$ in invariant on $D$ and $\tilde T$ is invariant on $\tilde D$ and 
\begin{equation*}
\Phi_t f= \sum_{n=1}^N\lambda_n T_tg_n \tilde T_t\tilde g_n,\quad f\in\mathcal L.
\label{eq:}
\end{equation*}
We now show that $\mathcal{L}$ belongs to the domain of the generator of $\{\Phi_t\}_{t\ge0}$.\\
It is enough to show that $D\tilde D$ belongs to the domain of the generator of $\Phi_t$ as the domain of a generator is closed under linear combinations.\\ 
To do so we show that  $t^{-1}(\Phi_t g\tilde g-g\tilde{g})$ converges to $\tilde g Ag+g\tilde A\tilde g$  as $t\to 0$. Compute
\begin{align*}
|\ t^{-1}(\Phi_t g\tilde g-g\tilde g)-\tilde g Ag-g\tilde A\tilde g|\le &|  t^{-1}(T_t g\tilde T_t\tilde g-g\tilde g) \pm t^{-1}T_tg\tilde g\pm T_tg \tilde A\tilde g-\tilde g Ag-g\tilde A\tilde g|\\
\le &|t^{-1} (T_t g\tilde T_t\tilde g -T_tg\tilde g)- T_tg \tilde A\tilde g|\\
&+|  t^{-1}(T_t g\tilde T_t\tilde g-g\tilde g) + t^{-1}T_tg\tilde g+ T_tg \tilde A\tilde g-\tilde g Ag-g\tilde A\tilde g|\\
\le &\|T_t g\|_{X}\|t^{-1} (\tilde T_t\tilde g -\tilde g)-\tilde A\tilde g\|_{\tilde X}\\
&+| - t^{-1}g\tilde g + t^{-1}T_tg\tilde g+ T_tg \tilde A\tilde g-\tilde g Ag-g\tilde A\tilde g|\\
\le &\|g\|_{X}\|t^{-1} (\tilde T_t\tilde g -\tilde g)-\tilde A\tilde g\|_{\tilde X}\\
&+\|\tilde g\|_{\tilde X}\|  t^{-1}(T_tg- g +)- Ag\|_{X}+\|\tilde A\tilde g\|_{\tilde X} \|T_tg -g\|_{X},
\label{eq:}
\end{align*}
which can be made arbitrarily small independently of $(x,\tilde x)\in X\times \tilde X$ by choosing $t$ small by strong continuity and the uniform boundedness  of $(T_t)$ and $(\tilde T_t)$ (here the notation $\|h\|_Y$ means the supremum norm of the function $h:Y\to\mathbb R$).\\
Therefore we have shown that $\mathcal{L}$ is a dense invariant subspace of $Dom(L)$ and by \cite[Proposition 1.9.1]{KV0} $\mathcal{L}$ is a core for the generator of $\Phi_t$, and $L=A+\tilde A$ on $\mathcal{L}$.\\

(ii). That the semigroup $\{\Phi_t\}_{t\ge 0}$ is a Feller semigroup if $\{T_t\}_{t\ge 0}$ and $\{\tilde T_t\}_{t\ge 0}$ are Feller semigroups follows easily. The same for the sub-Feller case.\\

(iii). The case of $C_{a,\infty}([a,b]\times \tilde X)$ has the same proof as above apart from the statement about the density of $\mathcal L$. Briefly, to obtain the density of the respective set $\mathcal L$, consider the linear span of the product of smooth functions in $C_{\infty}((-\infty,b]\times \tilde X)$, apply Stone-Weierstrass as above, then use an isometric isomorphism between $C_{a,\infty}((a,b]\times \tilde X)$ and $C_{a,\infty}((-\infty,b]\times \tilde X)$.

\subsection{Proof of Proposition \ref{3processesresults}-(i)}\label{proof3processesresults}

Fix $h>0$ and consider the \emph{bounded} generators $-D^{(\nu),h}_{+}$, $-D^{(\nu),h}_{a+}$ and $-D^{(\nu),h}_{a+\ast}$ defined as 
\begin{align*}
-D^{(\nu),h}_{+}f(x)&=\int_h^{\infty}(f(x-y)-f(x))\nu(x,y)dy,\\
-D^{(\nu),h}_{a+}f(x)&=\int_h^{\max\{(x-a), h\}}(f(x-y)-f(x))\nu(x,y)dy-f(x)\int_{\max\{(x-a), h\}}^\infty\nu(x,y)dy,\\
-D^{(\nu),h}_{a+\ast}f(x)&=\int_h^{\max\{(x-a), h\}}(f(x-y)-f(x))\nu(x,y)dy+(f(a)-f(x))\int_{\max\{(x-a), h\}}^\infty\nu(x,y)dy,
\end{align*}
acting on the spaces $C_\infty((-\infty,b])$,  $C_a([a,b])$, $C([a,b])$, respectively. Then 
\begin{equation}
T^{(\nu)+,h}_s=\sum_{n=0}^\infty \frac{t^n}{n!}(-D^{(\nu),h}_{+})^n,\quad T^{(\nu)a+,h}_s=\sum_{n=0}^\infty \frac{t^n}{n!}(-D^{(\nu),h}_{a+})^n,\quad  T^{(\nu)a+*,h}_s=\sum_{n=0}^\infty \frac{t^n}{n!}(-D^{(\nu),h}_{a+*})^n
,
\label{expform}
\end{equation}

$s\in \mathbb R^+$, are the respective semigroups. We first prove the second part of  Proposition \ref{3processesresults}-(i).\\
The key observation is that 
\[
-D^{(\nu),h}_{+}f(t)=-D^{(\nu),h}_{a+}f(t)=-D^{(\nu),h}_{a+\ast}f(t),\quad t>a,
\]
 if $f\in \{f(x)=0 \ \forall x\le a\}\cap C_\infty((-\infty,b])$, and
\begin{align*}
-D^{(\nu),h}_{+}f&\in \{f(x)=0 \ \forall x\le a\}\cap C_\infty((-\infty,b]),\\
-D^{(\nu),h}_{a+}f&\in \{f(x)=0 \ \forall x\le a\}\cap C_a([a,b]),\\
-D^{(\nu),h}_{a+*}f&\in \{f(x)=0 \ \forall x\le a\}\cap C_a([a,b]).
\end{align*}
 Hence for every $n\in\mathbb N$,
\begin{equation}
(-D^{(\nu),h}_{+})^nf(t)=(-D^{(\nu),h}_{a+})^nf(t)=(-D^{(\nu),h}_{a+\ast})^nf(t),\quad t>a,
\label{eqeq}
\end{equation}
if $f\in\{f(x)=0 \ \forall x\le a\}$.\\
The identities in (\ref{eqeq}) imply that   
\begin{equation}
T^{(\nu)+,h}_sf (t)= T^{(\nu)a+,h}_sf (t)=T^{(\nu)a+*,h}_sf (t),\quad t>a,\ s\in\mathbb R^+,
\label{semiequal}
\end{equation}
 if $f\in \{f(x)=0 \ \forall x\le a\}$, as each of the semigroups is given by the exponentiation formula in (\ref{expform}).  By the proofs of \cite[Theorem 5.1.1]{KV0} and \cite[Theorem 4.1]{KVFDE} 
\begin{equation}
T^{(\nu)+,h}_sf (t)\to T^{(\nu)+}_sf (t),\quad h\to 0,
\label{freeconv}
\end{equation}

 for each $s\ge 0$,  $t\in(a,b],$ $ f\in C^1_\infty((-\infty,b])$, and 
\begin{equation}
T^{(\nu)a+,h}_sf (t)\to T^{(\nu)a+}_sf (t),\quad T^{(\nu)a+*,h}_sf (t)\to T^{(\nu)a+*}_sf (t),\quad h\to 0,
\label{intconv}
\end{equation}
 for each $s\ge 0$,  $t\in(a,b],$ $ f\in C^1([a,b])\cap C_a([a,b])$. \\
Hence, If $f\in C^1((-\infty,b])\cap\{f(x)=0 \ \forall x\le a, f'(a)=0\}$, then (\ref{semiequal}) holds and we also have the convergence in (\ref{freeconv}) and (\ref{intconv}).\\
Now approximate point-wise from below the indicator function of any interval in $(a,b]$ with functions in $ C^1([a,b])\cap\{f(x)=0 \ \forall x\le a, f'(a)=0\}$ to obtain the second part of  Proposition \ref{3processesresults}-(i).\\

The first part of  Proposition \ref{3processesresults}-(i) follows similarly after observing that 
\[
T^{(\nu)+,h}_sf_y (t)= T^{(\nu)a+,h}_sf_y (t)=T^{(\nu)a+*,h}_sf_y (t)=0\quad \forall a<t\le y
\]
for any $f\in \{f(x)=0 \ \forall x\le y\}\cap C_\infty((-\infty,b])$. In the last step  we approximate   the indicator function $\mathbf 1(\cdot >y)$ with functions in $ C^1((-\infty,b])\cap\{f(x)=0 \ \forall x\le y\}$, $y>a$ to obtain that for every $s\in\mathbb R^+$
\[
0=T^{(\nu)+}_s\mathbf 1(\cdot>y) (t)=\mathbf P[X^{t,(\nu)}_{+}(s) >y]=\mathbf P[X^{t,(\nu)}_{a+} (s)>y]=\mathbf P[X^{t,(\nu)}_{a+*}(s) >y], \quad 
\]
if $t\le y$.

\let\oldbibliography\thebibliography
\renewcommand{\thebibliography}[1]{\oldbibliography{#1}
\setlength{\itemsep}{0pt}}


\begin{thebibliography}{99}


\bibitem{Leo01} Anh, V. V., Leonenko, N. N.  (2001),\emph{ Spectral analysis of fractional kinetic equations with random data}. J. Statist. Phys. 104, no. 5-6, 1349-1387.

\bibitem{meerspecneg16} Baeumer, B., Kov\'acs, M., Meerschaert, M. M., Schilling, R.,  Straka, P. (2016), \emph{Reflected spectrally negative stable processes and their governing equations}. Transactions of the American Mathematical Society, 368(1), 227-248.

\bibitem{Bae01} Baeumer, B., Meerschaert, M. M. (2001), \emph{Stochastic solutions for fractional Cauchy problems}. Fractional Calculus and Applied Analysis 4.4: 481-500.


\bibitem{Bae05} Baeumer, B., Kurita, S., Meerschaert, M. M. (2005),  \emph{Inhomogeneous fractional diffusion equations.} Fractional Calculus and Applied Analysis 8.4 (2005): 371-386.

\bibitem{Bazh98}Bazhlekova, E. (1998),\emph{ The abstract Cauchy problem for the fractional evolution equation.} Fractional Calculus and Applied Analysis 1.3: 255-270.


\bibitem{Bouchaud1990} Bouchaud, J.P., Georges, A. (1990), \emph{Anomalous diffusion in disordered media: statistical mechanism, models and physical applications}, Physics Reports, \textbf{195}, 127-293.

\bibitem{carpinteri1997} Carpinteri, A., Mainardi, F. (1997), \emph{ Fractals and Fractional Calculus in Continuum Mechanics}, CISM International Centre for Mechanical Sciences, Springer Verlag, Wien-New York.

\bibitem{Meetemp} Chakrabarty, A., Meerschaert, M. M. (2011), \emph{ Tempered stable laws as random walk limits}. Statistics \& Probability Letters, 81(8), 989-997.


\bibitem{kai} Diethelm, K. (2010), \emph{The Analysis of Fractional Differential Equations, An application-oriented exposition using differential operators of Caputo Type}, Lecture Notes in  Mathematics, v. 2004,  Springer.
\bibitem{dynkin1965} Dynkin, E. B. (1965), \emph{Markov processes}, Vol. I, Springer-Verlag.
\bibitem{edwards} Edwards, J. T., Ford, N. J., Simpson,  A. C. , (2001), \emph{The Numerical Solutions of Linear Multi-term Fractional Differential Equations: Systems of Equations}, Journal of Computational and Applied Mathematics, \textbf{148}, 401-418. 
\bibitem{kochubei} Eidelman, S. E., Kochubei, A. N. (2004), \emph{Cauchy problem for fractional differential equations}. Elsevier, Journal of differential equations, \textbf{199}, pp. 211-255.
\bibitem{EK} Ethier, S.N.,  Kurtz, T. G.  (1986), \emph{Markov processes. Characterization and Convergence}.Wiley Series in Probability and Mathematical Statistics, New York Chicester, Wiley. 
\bibitem{Evans} L. C. Evans, \emph{Partial Differential Equations}, Graduate Studies in mathematics, Vol 19, American Mathematical Society, 1997.

\bibitem{distributed2} Gorenflo, R., Luchko, Y., Stojanovic, M. (2013), \emph{Fundamental solution of a distributed order time-fractional diffusion-wave equation as probability density}, Fractional Calculus and Applied Analysis, Volume \textbf{16}, Number 2, pp. 297-316.
\bibitem{gorenflo98} Gorenflo, R., Mainardi, F. (1998), \emph{Fractional calculus and stable probability distributions}, Archive of Mechanics, \textbf{50} (3), 377-388.

\bibitem{KV-1} Hern\'andez-Hern\'andez, M.E., Kolokoltsov, V. N.  (2015), \emph{On the probabilistic approach to the solution of  generalized fractional differential equations of Caputo and Riemann-Liouville type}, Journal of Fractional Calculus and Applications, Vol. 7(1) Jan. 2016, pp. 147-175.
\bibitem{KV-2} Hern\'andez-Hern\'andez, M.E., Kolokoltsov, V. N.  (2015), \emph{Probabilistic solutions to non-linear fractional differential equations of generalized  Caputo and Riemann-Liouville type}, submitted for publication.

\bibitem{kilbas2} Kilbas, A. A., Srivastava, H. M., Trujillo, J. J. (2006), \emph{Theory and Applications of Fractional Differential Equations}, North-Holland Mathematics Studies, \textbf{204}, Elsevier. 

\bibitem{Wings}  Klafter, J., I. M. Sokolov, \emph{Anomalous Diffusion Spreads its Wings}, Physics World \textbf{18}, 29 August (2005).
\bibitem{Ko13}  Kochubei, A. N., (1990), \emph{Fractional-order diffusion}, Differential Equations 26, 485-492.
\bibitem{Kokine}  Kochubei,  A. N., Kondratiev, Y. (2017). \emph{Fractional kinetic hierarchies and intermittency}. Kinet. Relat. Models 10:3, 725 - 740.

\bibitem{KV} Kolokoltsov, V. N. (2009), \emph{Generalized continuous-time random walks (CTRW), Subordination by Hitting times and fractional dynamics.} Theory Probab. Appl. Vol. 53, No. 4, pp. 549-609.
\bibitem{KV0} Kolokoltsov, V. N. (2011), \emph{Markov processes, semigroups and generators.} DeGruyter Studies in Mathematics, Book 38. 
\bibitem{KVFDE} Kolokoltsov, V. N. (2015), \emph{On fully mixed and multidimensional extensions of the Caputo and Riemann-Liouville derivatives, related Markov processes and fractional differential equations},  Fractional Calculus and Applied Analysis, 18.4 (2015): 1039-1073.
\bibitem{KM-1} Kolokoltsov, V. N.,  Veretennikova, M. (2014), \emph{Well-posedness and regularity of the Cauchy problem for non-linear fractional in time and space equations}, Fractional Differential Calculus 4:1, 1-30.

\bibitem{Leo13} Leonenko, N. N., Meerschaert, M. M., Sikorskii, A. (2013), \emph{Fractional Pearson diffusions}. J. Math. Anal. Appl. 403, no. 2, 532-546.



\bibitem{LoHiBe2011} L\"orinczi, J\'ozsef, Hiroshima, Fumio; Betz, Volker \emph{Feynman-Kac-type theorems and Gibbs measures on path space. With applications to rigorous quantum field theory}. De Gruyter Studies in Mathematics, 34. Walter de Gruyter \& Co., Berlin, 2011. xii+505 pp.

\bibitem{FMainardi1997} Mainardi, F. (2001), \emph{Fractional calculus: some basic problems in continuum and statistical mechanics}, http://arxiv.org/abs/1201.0863v1.
\bibitem{FMainardi2010} Mainardi, F. (2010), \emph{Fractional Calculus and Waves in Linear Viscoelasticity. An introduction to Mathematical Models}, Imperial College Press. 

\bibitem{distributed1} Mainardi, F., Mura, A., Pagnini, G., Gorenflo, R. (2008), \emph{Time-fractional diffusion of distributed order}, J. Vib. Control \textbf{14}, pp. 1267-1290.


\bibitem{meshart} Meerschaert, M. M., Nane, E., Vellaisamy, P. (2009), \emph{Fractional Cauchy problems on bounded domains}. The Annals of Probability, Vol. 37, No. 3.

\bibitem{Meerschaert2012} Meerschaert, M.M., Sikorskii, A. (2012), \emph{Stochastic Models for Fractional Calculus}, De Gruyter Studies in Mathematics, Book \textbf{43}.

\bibitem{non1990} Nonnenmacher, T. F. (1990),  \emph{Fractional integral and  differential equations for a class of Levy-type probability densities}, J. Phys. A: Math. Gen. \textbf{23}.



\bibitem{Nane} E. Nane (2010), \emph{Fractional Cauchy problems on bounded domains: survey of recent results},  Fractional Dynamics and Control (2011): 185.



\bibitem{podlubny}  Podlubny, I. (1999), \emph{Fractional differential equations.  An introduction to fractional derivatives, fractional differential equations, to methods of their solution and some of their applications.} Mathematics in Science and Engineering, v. 198. Academic Press, Inc., San Diego.

\bibitem{samko} Samko, S. G., Kilbas, A. A., Marichev, O. I. (1993), \emph{Fractional integrals and derivatives: theory and applications}, Gordon and Breach Science Publishers S. A.
\bibitem{scalas} Scalas, E. (2012), \emph{A class of CTRW's: compound fractional Poisson processes. In: Fractional dynamics, World Sci. Publ., Hackensack, NJ, pp. 353-374}.
\bibitem{Wyss22}  W. R. Schneider and W. Wyss, \emph{Fractional diffusion and wave equations}, J. Math. Phys. 30 (1989), 134-144.
\bibitem{willard70} Willard, S., \emph{General Topology}, Addison-Wesley Series in Mathematics, (1970).
\bibitem{WyTemp}Wy\l{}oma\'nska, A. (2013),\emph{ The tempered stable process with infinitely divisible inverse subordinators}. Journal of Statistical Mechanics: Theory and Experiment, 2013(10), P10011.


\bibitem{Zhang00} Zhang, Y., (2000), \emph{Sufficient and necessary conditions for stochastic comparability of jump processes.} Acta Mathematica Sinica 16.1 : 99-102.

\bibitem{zaslavsky} Zaslavsky, G. M. (2002) \emph{Chaos, fractional kinetics, and anomalous transport}, Physics Reports, \textbf{371}, pp. 461-580.
\bibitem{zolotarev}  Zolotarev, V. M. (1986) \emph{One-dimensional stable distributions}. Translations of Mathematical monographs, vol. 65, American Mathematical Society, 1986. 








\end{thebibliography}
\end{document}